\documentclass{amsart}
\usepackage{amsfonts,amssymb,latexsym}
\usepackage{amsmath,amsthm}
\usepackage{eucal}
\usepackage[latin1]{inputenc}
\usepackage{color}

\newtheorem{theorem}{Theorem}[section]
\newtheorem{lemma}[theorem]{Lemma}
\newtheorem{proposition}[theorem]{Proposition}
\newtheorem{corollary}{Corollary}[section]

\theoremstyle{definition}
\newtheorem{definition}{Definition}[section]

\theoremstyle{remark}
\newtheorem{remark}{Remark}[section]
\def\R{{\mathbb R}}
\def\N{{\mathbb N}}

\def\supp{\mathop{\rm supp}\nolimits}

\def\M{{\mathcal M}}
\newcommand{\dist}[2]{\Bigl\langle #1, #2 \Bigr\rangle}

\numberwithin{equation}{section}
\newcommand{\tendsto}[1]{\renewcommand{\arraystretch}{0.5}
\begin{array}[t]{c}
\longrightarrow \\
{ \scriptstyle #1 }
\end{array}
\renewcommand{\arraystretch}{1}}

\newcommand{\weaktendsto}[1]{\renewcommand{\arraystretch}{0.5}
\begin{array}[t]{c}
\rightharpoonup \\
{ \scriptstyle #1 }
\end{array}
\renewcommand{\arraystretch}{1}}

\begin{document}
\title[A Liouville property for the Camassa-Holm equation]{A Liouville  property with application to   asymptotic stability for the Camassa-Holm equation}


\author[L. Molinet]{Luc Molinet}

\thanks{L.M. was partially supported by the french ANR project GEODISP}
\begin{abstract}
We prove a Liouville property for uniformly almost localized (up to translations) $ H^1$-global solutions of the Camassa-Holm equation with a momentum density that is a non negative finite measure.
More precisely, we show that  such solution has to be a  peakon. \\
As a consequence, we prove that peakons are asymptotically stable in the class of $ H^1$-functions with a momentum density that belongs to $ {\mathcal M}_+(\R) $.
Finally, we also  get an  asymptotic stability result for  train of peakons.
 \end{abstract}

\address{Luc Molinet, Institut Denis Poisson, Universit\'e de Tours, Universit\'e d'Orl\'eans, CNRS,  Parc Grandmont, 37200 Tours, France.}
\email{Luc.Molinet@univ-tours.fr}

\date{\today}

\maketitle

\section{Introduction}
\noindent The  Camassa-Holm equation (C-H),
\begin{equation}
u_t -u_{txx}=- 3 u u_x +2 u_x u_{xx} + u u_{xxx}, \quad
(t,x)\in\R^2, \label{CHCH}\\
\end{equation}
can be derived as a model for the propagation of unidirectional
shalow water waves over a flat bottom  by  writing the
Green-Naghdi equations in Lie-Poisson Hamiltonian form and then
making an asymptotic expansion which keeps the Hamiltonian
structure (\cite{CH1}, \cite{Johnson}). 
A rigorous derivation  of the Camassa-Holm equation from the full water waves problem is obtained in    \cite{AL} and  \cite{CL}.

(C-H) is completely integrable (see \cite{CH1},\cite{CH2},
\cite{C1} and \cite{CGI}) and enjoys also a geometrical derivation (cf. \cite{K1}, \cite{K2}). It possesses among  others the
following invariants
\begin{equation} 
M(v)= \int_{\R} (v-v_{xx}) \, dx , \; E(v)=\int_{\R} v^2(x)+v^2_x(x)
\, dx \mbox{ and } F(v)=\int_{\R} v^3(x)+v(x)v^2_x(x) \, dx\;
\label{E}
\end{equation}
and can be written in Hamiltonian form as
\begin{equation}
\partial_t E'(u) =-\partial_x F'(u) \quad .
\end{equation}
It is also worth noticing that \eqref{CH} can be rewritted as
\begin{equation}\label{CHy}
y_t +u y_x+2 u_x y =0 
\end{equation}
 which is a transport equation for the momentum density  $y=u-u_{xx} $.

Camassa and Holm \cite{CH1} exhibited peaked solitary waves
solutions to (C-H) that  are
 given by
$$ u(t,x)=\varphi_c(x-ct)=c\varphi(x-ct)=ce^{-|x-ct|},\; c\in\R^*.$$
They are called peakon whenever $ c>0 $ and antipeakon  whenever
$c<0$.  
Note that the initial value problem associated with (C-H) has to be rewriten as
\begin{equation}
\left\{ \begin{array}{l}
u_t +u u_x +(1-\partial_x^2)^{-1}\partial_x (u^2+u_x^2/2)=0\\
\label{CH} \; 
u(0)=u_0, 
\end{array}
\right.
\end{equation}
to give a meaning to these solutions.

Their stability seems not
to enter the general framework developed
 for instance in \cite{Benjamin}, \cite{GSS}, especially because of the non smoothness of the peakon. However, Constantin and
Strauss \cite{CS1} succeeded in proving their orbital stability by
a direct approach.  This approach is based on two  optimal inequalities: one involving $ E(u-\varphi) $ and  $\max_{\R}u $ and  the other one involving $ E(u) $, $F(u) $ and $\max_{\R} u $. 

In a series of papers (see for instance \cite{MM1}, \cite{MM2}) Martel and Merle developped  an approach, based on a Liouville property for uniformly almost localized  global solutions 
close to the solitary waves, to prove the asymptotic stability for a wide class of dispersive equations. The Liouville property is based on the study of a dual equation related to the linearized equation around the solitary waves. Such approach to prove the Liouville property seems not to be applicable for the C-H equation. Indeed, first working with the dual problem requires more regularity on the solution and, in contrast to  KdV-like  equations, one cannot require the asymptotic objects
 of the C-H equation to be smooth (see  the peakon !). Second, for the same reasons for which there is no proof of the orbital stability by the spectral method, it seems very difficult to  get a non negative property on the underlying linear operator.

 In this paper we prove a Liouville result for uniformly almost localized (up to translations) global solutions to the CH equation and then follows the general strategy developed by Martel and Merle to deduce the asymptotic stability of the peakon. The main ingredient to prove our Liouville result is   the finite speed  propagation of  the momentum density of the  solution. We would like to underline that our arguments are not specific to the Camassa-Holm equation but can be adapted for a wide class of equations with peakons as we will show in a forthcoming work.

In this paper we will work in the framework of the solutions constructed in \cite{CM1}.  This class corresponds to solutions emanating from initial data that belong to $ H^1(\R) $ with a density momentum that is a non negative finite measure. It has the advantage  to contain the peakon and to enjoy good properties as global existence, uniqueness and $ H^1$-continuity of the flow. 

  It is worth noticing that, recently\footnote{See for instance \cite{L} for a survey on previous existence results.}, Bressan and Constantin (\cite{BC1}, \cite{BC2}) succeed to  construct global conservative and dissipative solutions of the \eqref{CH} for initial data in $ H^1(\R) $ by using scalar conservation laws techniques. The uniqueness of the conservative solution has been shown very recently by Bressan-Chen-Zhang \cite{BCZ} but its continuity  with values in $ H^1(\R) $ is not known. In consequence, the orbital stability of the peakon with respect to these solutions is still an open problem. In this direction, note  that a Lipschitz metric on $ H^1 $-bounded sets has been very recently constructed in \cite{CCCS}.

Before stating our results let us introduce the function space to which  our initial data belong. Following \cite{CM1}, we introduce the following space of functions
\begin{equation}
Y=\{ u\in H^1(\R)    \mbox{ such that  }  u-u_{xx} \in {\mathcal M}(\R)  \} \; 
\end{equation}
where $ {\mathcal M}(\R) $ is the space of finite Radon measures on $ \R $.
We denote by $ Y_+ $ the closed subset of $ Y $ defined by  $Y_+=\{u\in Y \, /\, u-u_{xx}\in \M_+ \} $ where $ {\mathcal M}_+ $ is the set of non negative finite Radon measures on $ \R$.

Let $ C_b(\R)$  be the set of bounded continuous functions on $ \R $, $ C_0(\R)$  be  the set of continuous functions  on $ \R $ that tends to $ 0 $ at infinity and let $ I \subset \R$ be an interval.  
 A sequence $\{\nu_n\}\subset {\mathcal M} $ is said to converge tightly (resp. weakly) towards $ \nu\in {\mathcal M} $ if for any $ \phi\in C_b(\R) $ (resp. $C_0(\R)$), $ \langle \nu_n,\phi\rangle \to
  \langle \nu,\phi\rangle $. We will then write $ \nu_n  \rightharpoonup \! \ast \; \nu $ tightly  in $ \M $ (resp. $ \nu_n  \rightharpoonup \! \ast \; \nu $ in $\M$).
 
 Throughout this paper, $ y\in C_{ti}(I;\M) $ (resp.   $ y\in C_{w}(I;\M) $) will signify that for any $ \phi\in C_b(\R) $
  (resp. $\phi\in C_0(\R)$) , 
 $ t\mapsto \dist{y(t)}{\phi} $ is continuous on $ I$ and $ y_n \rightharpoonup \! \ast \; y $ in $ C_{ti}(I;\M) $ (resp. $ y_n \rightharpoonup \! \ast \; y $ in $ C_{w}(I;\M) $) will signify that for any $ \phi\in C_b(\R) $ (resp. $C_0(\R)$),
 $ \dist{y_n(\cdot)}{\phi}\to \dist{y(\cdot)}{\phi} $ in $ C(I)$. 

\begin{definition}\label{defYlocalized}
 We say that  a solution $ u \in C(\R; H^1(\R)) $ with $ u-u_{xx}\in C_{w}(\R; \M_+) $ of \eqref{CH} is $ Y$-almost localized if there exist $ c>0 $ and a $ C^1 $-function $ x(\cdot) $, with $ \dot{x}\ge c>0 $,  for which for any $ \varepsilon>0 $, there exists $ R_{\varepsilon}>0 $ such that for all $ t\in\R $ and all $ \Phi\in C(\R) $ with $0\le \Phi\le 1 $ and $ \supp \Phi \subset [-R_\varepsilon,R_\varepsilon]^c $.
  \begin{equation}\label{defloc}
 \int_{\R} (u^2(t)+u_x^2(t))  \Phi(\cdot-x(t)) \, dx + \Bigl\langle  \Phi(\cdot-x(t)), u(t)-u_{xx}(t)\Bigr\rangle \le \varepsilon \; .
\end{equation}
 \end{definition}

\begin{theorem}\label{liouville}
 Let $ u \in C(\R; H^1(\R)) $,  with $ u-u_{xx}\in C_{w}(\R; \M_+) $, be a $ Y$-almost localized solution of \eqref{CH} that is not identically vanishing. Then there exists $ c^*	>0 $ and $ x_0\in \R $ such that  $$
 u(t)=c^* \, \varphi(\cdot -x_0-c^* t) , \quad \forall t\in \R \, .
 $$
\end{theorem}
\begin{remark}
This theorem implies, in particular, that a $Y$-almost localized solution with non negative momentum density cannot be smooth for any time. More precisely, 
 if $ u\in C(\R;H^1) $, with $ u-u_{xx}\in C_{w}(\R; \M_+) $, is  a  $ Y$-almost localized solution of the Camassa-Holm equation that   belongs  to $ H^{\frac{3}{2}}(\R) $ for some $ t\in \R $ then $ u $  must be the trivial null solution.
\end{remark}

As a consequence we get the asymptotic stablity of the peakons :
\begin{theorem} \label{asympstab} 
Let $ c>0 $ be fixed. There exists a  universal constant $0<\eta_0\ll 1  $ such that for any $ 0<\theta<c $ and any $ u_0\in Y_+ $ satisfying 
\begin{equation}\label{difini}
\|u_0-\varphi_c \|_{H^1} \le \eta_0 \Bigl(\frac{\theta}{c}\Bigr)^{8}\; ,
\end{equation}
there exists $ c^*>0 $ with $ |c-c^*|\ll c $ and a $C^1$-function $ x \, : \, \R \to \R $ 
 with $ \displaystyle\lim_{t\to \infty} \dot{x}=c^* $  such that
\begin{equation}
u(t,\cdot+x(t)) \weaktendsto{t\to +\infty} \varphi_{c^*} \mbox{ in } H^1(\R) \; ,
\end{equation}
where $ u\in C(\R; H^1) $ is the solution emanating from $ u_0 $.
Moreover,
\begin{equation}\label{cvforte}
\lim_{t\to +\infty} \|u(t) -\varphi_{c^*}(\cdot-x(t))\|_{H^1(]\theta t ,+\infty[)}=0 \; .
\end{equation}
\end{theorem}
\begin{remark} Using that \eqref{CH} is invariant by the change of unknown $ u(t,x)\mapsto -u(t,-x) $, we obtain as well the asymptotic stability of the antipeakon profile $ c \varphi $ with $ c<0 $ in the class of $ H^1$-function with a momentum density that belongs to $ \M_-(\R) $.
\end{remark}
\begin{remark} This theorem implies the growth of the  high Sobolev norms for some smooth solutions of the Camassa-Holm equation. 
Indeed, it is proven in    \cite{CE1} that any initial datum $ u_0 \in H^\infty(\R) \cap Y_+ $ gives rise to a solution $u \in C(\R;H^\infty(\R)) $ and the above theorem
 ensures that if such initial datum satisfies
 \eqref{difini},  then $ \|u(t)\|_{H^{3/2}} \to +\infty $ as $ t \to +\infty $.
\end{remark}
\begin{remark} Theorem \ref{asympstab}  and especially Theorem \ref{asympt-mult-peaks}, in Section \ref{6},  are first steps towards the peakons decomposition of solutions with a non negative momentum density that is studied by the inverse scattering approach in \cite{eck}.
\end{remark}
This paper is organized as follows : in the next section we recall the well-posedness results for the class of solutions we will work with and, in Section 3, we derive an almost monotonicity result that implies an  exponential decay result for $ Y$-almost localized global solutions. Section 4 is devoted to the Liouville theorem for $ Y$-almost localized solutions which is the heart of this work. Finally, in Section 5 and 6, we respectively prove the asymptotic stability of a single peakon and of a train of peakons. 
\section{Global well-posedness results} 
We first recall some obvious estimates that will be useful in the sequel of this paper.  
Noticing that $ p(x)=\frac{1}{2}e^{-|x|} $ satisfies $p\ast y=(1-\partial_x^2)^{-1} y $ for any $y\in H^{-1}(\R) $ we easily get
$$
\|u\|_{W^{1,1}}=\|p\ast (u-u_{xx}) \|_{W^{1,1}}\lesssim \| u-u_{xx}\|_{\M}
$$
and 
$$
\|u_{xx}\|_{\M}\le \|u\|_{L^1}+ \|u-u_{xx} \|_{\M} 
$$
 which ensures  that 
\begin{equation} \label{bv}
Y\hookrightarrow  \{ u\in W^{1,1}(\R) \mbox{ with } u_x\in {\mathcal BV}(\R) \} \; .
\end{equation}
 It is also worth noticing that since for $ v\in C^\infty_0(\R) $, 
$$
v(x)=\frac{1}{2} \int_{-\infty}^x e^{x'-x} (v-v_{xx})(x') dx' +\frac{1}{2} \int_x^{+\infty} e^{x-x'} (v-v_{xx})(x') dx'
$$
and 
$$
v_x(x)=-\frac{1}{2}\int_{-\infty}^x e^{x'-x} (v-v_{xx})(x') dx' + \frac{1}{2}\int_x^{+\infty} e^{x-x'} (v-v_{xx})(x') dx' \; ,
$$
we get $ v_x^2 \le v^2 $ as soon as $ v-v_{xx} \ge 0 $ on $ \R $. By the density of $ C^\infty_0(\R) $ in $ Y $, we deduce that 
\begin{equation}\label{dodo}
 |v_x|\le v \mbox { for any } v\in Y_+ \; .
\end{equation}
Finally, throughout this paper, we will denote $ \{\rho_n\}_{n\ge 1} $ the mollifer defined by 
\begin{equation} \label{rho}
\rho_n=\Bigl(\int_{\R} \rho(\xi) \, d\xi 
\Bigr)^{-1} n \rho(n\cdot ) \mbox{ with } \rho(x)=\left\{ 
 \begin{array}{lcl} e^{1/(x^2-1)} & \mbox{for} & |x|<1 \\
0 & \mbox{for} & |x|\ge 1 
\end{array}
\right.
\end{equation}
In \cite{CE1}  a global well-posedness result is shown for smooth solution to \eqref{CH} with a non negative momentum density. This result can be summarized in the following proposition 
 \begin{proposition}{(Global smooth solutions \cite{CE1})}\label{smoothWP} \\
 Let $ u_0\in H^3(\R)$, satisfying $y_0= u_{0}-u_{0,xx}\ge 0 $ with $ y_0\in L^1(\R) $, then the initial value problem associated with \eqref{CH}  has a unique solution 
$u\in C(\R;H^3(\R))\cap C^1(\R;H^2(\R)) $. This solution satisfies  $y=u-u_{xx} \ge 0 $ on $ \R^2$ and  $ M(\cdot) $, $E(\cdot) $ and $ F(\cdot) $ are constant along the trajectory
 \footnote{In particular, $ y\in L^\infty(\R;L^1(\R)) $.}.  If moreover $ u_0\in H^\infty(\R) $ then $u\in C(\R;H^\infty(\R)) $.
\end{proposition}

Unfortunately, the peakons do not enter in this framework since their profiles do not belong  even to $ H^{\frac 3 2}(\R) $. In \cite{CM1} an existence and uniqueness result of global solutions to \eqref{CH} in  a class of functions that contains the peakon is proved. This result will be crucial in our analysis. We give below a slightly improved version:
 \begin{proposition}{(Global weak solution \cite{CM1})}\label{WP} \\
 Let $ u_0 \in Y_+ $ be given. \vspace*{2mm} \\
 {\bf 1. Uniqueness and global existence :} \eqref{CH} has a unique solution $ u\in C^1(\R; L^2(\R))\cap C(\R;H^1(\R)) $ such that 
  $ y=(1-\partial_x^2)u \in C_{ti}(\R; \M_+) $. Moreover, $E(u) $ $ F(u) $ and $ M(u)=\dist{y}{1} $ are conservation laws . \vspace*{2mm} \\
  {\bf 2. Continuity with respect to initial data  in $H^1(\R)$}: For any sequence $ \{u_{0,n}\} $ bounded in $ Y_+ $ such that $ u_{0,n} \to u_0 $ in $ H^1(\R ) $
   and $ (1-\partial_x^2)u_{0,n}   \rightharpoonup  \! \ast  u_0-u_{0,xx}$ tightly in $ \M$,   the emanating sequence of solution $ \{u_n\} \subset  C^1(\R_+;L^2(\R))\cap C(\R_+;H^1(\R)) $ satisfies for any $ T>0 $
  \begin{equation}\label{cont1}
  u_n \to u \mbox{ in } C([-T,T]; H^1(\R))  \end{equation}
  and 
    \begin{equation}\label{cont2}
 (1-\partial_x^2) u_n  \rightharpoonup  \! \ast \; y \mbox{ in } C_{ti} ([-T,T], \M) \; . 
   \end{equation}
       {\bf 3. Continuity with respect to initial data  in $Y$ equipped with the weak topology}: For any sequence $ \{u_{0,n}\} \subset Y_+ $ such that\footnote{By this we mean that $ u_{0,n} 
     \rightharpoonup u_0$ in $ H^1(\R)     $ and $ (1-\partial_x^2) u_{0,n}  \rightharpoonup \! \ast  \; u_0-u_{0,xx}  $ in $ \M$} $ u_{0,n} \rightharpoonup \! \ast \; u_0 $ in $ Y $,  the emanating sequence of solution $ \{u_n\} \subset  C^1(\R;L^2(\R))\cap C(\R_+;H^1(\R)) $ satisfies for any $ T>0 $,
  \begin{equation}\label{weakcont}
  u_n \weaktendsto{n\to\infty} u \mbox{ in } C_{w}([-T,T]; H^1(\R) ) \; ,
  \end{equation}
  and
   \begin{equation}\label{weakcont2}
 (1-\partial_x^2) u_n  \rightharpoonup  \! \ast \; y \mbox{ in } C_{w} ([-T,T], \M)  \; . 
   \end{equation}
  \end{proposition}
  \begin{proof}
  The uniqueness and global existence results are obtained in \cite{CM1} except the conservation of $ M(u) $ and the fact that $ y $ belongs to $ C_{ti}(\R; \M_+)$. 
   In \cite{CM1}, only the fact that $ y\in L^\infty(\R,\M_+) $ is stated. Note also that in \cite{CM1} the results  are stated only for positive times but, since the equation is reversible with time, it is direct to check that the results hold as well for negative times. 
  
 To prove \eqref{cont1}, it suffices to notice that, according to the conservation of the $ H^1$-norm   and \eqref{bv}, the sequence of emanating solution $\{u_n\} $ is bounded in  $ C(\R_+;H^1(\R))\cap L^\infty(\R; W^{1,1}(\R)) $
  with $ \{u_{n,x} \} $ bounded in $ L^\infty(\R; {\mathcal BV}(\R)) $.  Therefore, there exists $ v\in  L^\infty(\R;H^1(\R)) $ with $ (1-\partial_x^2) v \in 
   L^\infty(\R; {\mathcal M_+(\R)}) $ and an increasing sequence of integers $ \{n_k\}_{k\ge 1} $ such that,  for any $ T>0$,
  $$
     u_{n_k} \weaktendsto{k\to\infty} v \in L^\infty([-T,T]; H^1(\R))  \mbox{ and }  (1-\partial_x^2) u_{n_k} \weaktendsto{k\to\infty} \hspace*{-3mm} \ast \; (1-\partial_x^2) v  \mbox{ in } L^\infty(]-T,T[; {\mathcal M}_+(\R))  
$$
   But, using that $ \{\partial_t u_n\} $ is bounded in $L^\infty(\R; L^2(\R) \cap L^1(\R) )$, Helly's,  Aubin-Lions' compactness and Arzela-Ascoli's theorems then ensure that $ v $ is a solution to \eqref{CH} that belongs to $ C_{w}([-T,T]; H^1(\R)) $ with  $ v(0)=u_0 $ and that 
    \begin{equation} \label{yt}
    (1-\partial_x^2) u_{n_k}  \rightharpoonup  \! \ast \; (1-\partial_x^2) v \mbox{ in } C_{w} ([-T,T], \M) \; .
    \end{equation} 
     In particular, $ v_t\in L^\infty(]-T,T[; L^2(\R)) $ and thus $ v\in C([-T,T];L^2(\R)) $. Since $ v\in L^\infty(]-T,T[; H^{\frac{3}{2}-} (\R))$, this  actually implies that 
   $ v\in C([-T,T];  H^{\frac{3}{2}-}(\R)) $ and, using again the equation, it follows that $ v_t\in C(\R; L^2(\R)) $. Therefore, 
 $v$  belongs to the uniqueness class which ensures that $ v=u$ and thus the above weak convergence results hold for the whole sequence $\{u_n\} $. 
 The conservation of $ E(\cdot) $ together with these weak convergence results then lead to \eqref{cont1}. 
 
 Let us now prove  that $M(\cdot)$ is a conservation law  for our solutions and that $y\in C_{ti}(\R; \M_+) $. For this we apply the same arguments as above but for a smooth sequence $ \{\tilde{u}_{0,n}\} \subset H^3(\R) \cap Y_+ $ that converges in the same sense to $u_0\in Y_+ $. According to Proposition \ref{smoothWP},  $ M(\cdot) $ is a conservation law for the solutions emanating from $ \tilde{u}_{0,n} $ and, by hypothesis, 
  $M(\tilde{u}_{0,n})\to M(u_0) $. Therefore, the convergence result \eqref{yt} ensures that $ M(u(t)) \le M(u_0) $ for all $ t\neq 0 $. But, approximating in the same way $ u(t_0) $ for $t_0\neq 0 $, the same arguments lead to $ M(u(t))\le M(u(t_0)) $ for all $ t\neq t_0 $ which forces  $ M(\cdot) $ to be a conservation law. Hence,  $ M(\tilde{u}_n(t)) \to M(u(t)) $ for all $t\in \R $. Since $  \{(1-\partial_x^2) \tilde{u}_n(t)\}\subset {\mathcal M}_+ $, it is well-known (see for instance Proposition 9, page 61, in \cite{Bourbaki}) that this convergence  result  together with the weak convergence  \eqref{yt}  ensure the tight convergence of  $ 
 ( 1-\partial_x^2) \tilde{u}_n(t) $ towards $y(t) $ for all $ t\in\R $. Using again that $ \{\partial_t \tilde{u}_n\} $ is bounded in 
 $L^\infty(]-T,T[;L^1(\R)) $ and Arzela-Ascoli's theorem, we obtain that $y\in  C_{ti}(\R;\M_+) $ and that \eqref{cont2} holds for  $ \{\tilde{u}_n\} $.
  
  Now, coming back to the sequence $ \{u_n\} $, we deduce from the tight convergence of $\{(1-\partial_x^2) u_{0,n}\} $ towards $ u_0-u_{0,xx} $ together with the conservation of $ M(\cdot) $  that $ M(u_n(t)) \to M(u(t)) $ for all $t\in \R $ and the same arguments as above lead to \eqref{cont2}.
 
    Finally \eqref{weakcont} can be proven exactly in the same way, since  $ \{u_{0,n}\} $ is bounded in $ Y_+ $  by  Banach-Steinhaus theorem. 
  \end{proof}
  \begin{remark}
  {\bf 3.} of Proposition \ref{WP} ensures that \eqref{CH} is a dynamical system in $ Y_+ $ endowed with its natural weak star topology, i.e.
  $$
  \varphi_n \rightharpoonup\!\ast\; \varphi \mbox{ in } Y \mbox{ iff }    \varphi_n \rightharpoonup \varphi \mbox{ in } H^1(\R) \mbox{ and }
  (1-\partial_x^2) \varphi_n\ \rightharpoonup\! \ast \;   (1-\partial_x^2)\varphi \text{ in } {\mathcal M}(\R) \; .
  $$
  \end{remark}

 \section{Decay of $Y$-almost localized solution moving to the right}
  \begin{proposition}\label{prodecay}
 Let $ u \in C(\R; H^1) $  with  $ y=(1-\partial_x^2)u \in C_{w}(\R; \M+) $ be a $Y$-almost localized solution of \eqref{CH} with $ \inf_{\R} \dot{x}\ge c_0>0$.  Then there exists $ C>0 $  such that for all $ t\in\R $, all 
  $ R >0 $  and all $ \Phi\in C(\R) $ with $0\le \Phi\le 1 $ and $ \supp \Phi \subset [-R,R]^c $.
\begin{equation}\label{estimatedecay}
 \int_{\R} (u^2(t)+u_x^2(t))  \Phi(\cdot-x(t)) \, dx + c_0 \dist{ \Phi(\cdot-x(t))}{y(t)} \le C \, \exp( -R/6) \; .
 \end{equation}
\end{proposition}
To prove this proposition, the main tool is an almost monotonicity result for $E(u)+c_0 M(u) $ at the right of an almost  localized solution. Actually, the almost monotonicity is more general and  says somehow that if $z(t) $ moves to the right with a positive speed strictly less that $ \dot{x}(t) $ then the part of $E(u)+c_0 M(u) $ at the right of $ z(t) $ is almost decreasing as soon as $ |z(t)-x(t)| $ stays large enough. 

As in \cite{MM2}, we introduce the $ C^\infty $-function $ \Psi $ defined  on $ \R $ by 
\begin{equation}
\Psi(x) =\frac{2}{\pi} \arctan \Bigl( \exp(x/6)\Bigr) 
\end{equation}
It is easy to check that $ \Psi(-\cdot)=1-\Psi $ on $ \R $, $ \Psi' $ is a  positive even  function and that 
 there exists $C>0 $ such that $ \forall x\le 0 $, 
\begin{equation}\label{psipsi}
|\Psi(x)| + |\Psi'(x)|\le C \exp(x/6) \; .
\end{equation}
Moreover, by direct calculations, it is easy to check that 
\begin{equation}\label{psi3}
|\Psi^{'''}| \le  \frac{1}{2} \Psi' 
\end{equation}
\begin{lemma}\label{almostdecay}
  Let $0<\alpha < 1 $ and let $ u\in C(\R;H^1) $,  with $ y=(1-\partial_x^2)u \in C_{w}(\R; \M_+) $, be a solution of \eqref{CH} such that there exist $x\,:\, \R\to \R $ of class $ C^1 $ with 
   $ \inf_{\R} \dot{x}\ge c_0>0$  and $ R_0>0 $ with
 \begin{equation}\label{loc}
 \|u(t)\|_{L^\infty(|x-x(t)|>R_0)} \le \frac{(1-\alpha) c_0}{2^6} \, , \; \forall t\in\R .
 \end{equation}
  For $ 0<\beta \le \alpha $, $ 0\le \gamma\le \frac{3}{2} (1-\alpha) c_0 $, $ R>0 $, $ t_0\in\R $  and any $ C^1 $-function 
  \begin{equation}\label{condz}
  z\, :\, \R\to \R \mbox{ with } (1-\alpha)  \dot{x}(t) \le \dot{z}(t) \le (1-\beta) \dot{x}(t), \quad \forall t\in \R,
   \end{equation}
  setting
     \begin{equation}\label{defI}
 I^{\mp R}_{t_0} (t)=\dist{u^2(t)+u_x^2(t)+\gamma  y(t)}{\Psi\Bigl(\cdot- z_{t_0}^{\mp R}(t)\Bigr)}
 \end{equation}
  where 
 $$
 z_{t_0}^{\mp R}(t)=x(t_0)\mp R +z(t)-z(t_0)
 $$
 we have 
 \begin{equation}
I^{+R}_{t_0}(t_0)-I^{+R}_{t_0}(t)\le K_0 e^{-R/6} , \quad \forall t\le t_0 \quad  \label{mono}
\end{equation}
and 
 \begin{equation}
I^{-R}_{t_0}(t)-I^{-R}_{t_0}(t_0)\le K_0 e^{-R/6} , \quad \forall t\ge t_0 \quad , \label{mono2}
\end{equation}
for some constant $ K_0>0 $ that only depends on $ E(u) $, $c_0$, $R_0$ and $ \beta$. 
\end{lemma}
\begin{proof}
We first approximate $u(t_0) $ by the sequence of smooth functions $ u_{0,n}=\rho_n\ast u(t_0) $, with $ \{\rho_n\} $ defined in \eqref{rho}, that belongs to $H^\infty(\R) \cap Y_+ $
 and converges\footnote{By this me mean that $u_{0,n} \to u(t_0) $ in $ H^1(\R) $ and $(1-\partial_x^2)u_{0,n} \to  (1-\partial_x^2)u(t_0)$ in $\M$.}  to $ u(t_0) $ in $ Y$. According to Propositions \ref{smoothWP} and \ref{WP}, the  sequence of solutions $ \{u_n\} $ to \eqref{CH}  with $ u_n(t_0)=u_{0,n} $  belongs to $ C(\R;H^\infty(\R)) $ and 
 for any fixed $ T>0 $ it holds 
\begin{eqnarray}
u_n & \to & u \mbox{ in } C([t_0-T,t_0+T];H^1) \label{cvH1}\\
y_n & \rightharpoonup \! \ast & y  \mbox{ in } C_{ti} (]t_0-T,t_0+T[;{\mathcal M}) \label{weakcv1}
\end{eqnarray}
where $ y_n=u_n-\partial^2_x u_n $. In particular, for any fixed $ T>0 $, there exists $n_0=n_0(T)\ge 0 $ such that for any $ n\ge n_0 $, 
$$
\|u_n-u\|_{L^\infty(]t_0-T,t_0+T[\times \R)} <  \frac{(1-\alpha) c_0}{2^6} \; ,
$$
which together with  \eqref{loc} forces 
\begin{equation}\label{dif}
\sup_{t\in ]t_0-T,t_0+T[}  \|u_n\|_{L^\infty(|x-x(t)|>R_0)} <  \frac{(1-\alpha) c_0}{2^5} \; .
\end{equation}
We first prove that    \eqref{mono}  holds on $ [t_0-T,t_0] $ with  $ u $ replaced by $ u_n $ for $ n\ge n_0 $. 
The following computations hold for $ u_n $ with $ n\ge n_0$ but , to simplify the notation, we drop the index $ n $.  For any function $ g\in C^1(\R) $ it is not too hard to check that  (see Appendix \ref{sect71})
\begin{eqnarray}
\frac{d}{dt}\int_{\R} (u^2+u_x^2) g &=&\int_{\R} u u_x^2 g'  +2 \int_{\R} u h g'
\label{go}
\end{eqnarray}
where $ h:=(1-\partial_x^2)^{-1} (u^2+u_x^2/2) $. Moreover, it easily follows from \eqref{CHy} that 
 \begin{eqnarray}
 \frac{d}{dt}\int_{\R}  y g \, dx & = &-\int_{\R} \partial_x (y u) g - \int_{\R} y u_x g \nonumber \\
 & = & \int_{\R} y u g' -\int_{\R} (u-u_{xx}) u_x g \nonumber \\
  &= &  \int_{\R} y u g' +\frac{1}{2} \int_{\R} (u^2-u_x^2) g'  \label{go2}
 \end{eqnarray}
  Applying \eqref{go}-\eqref{go2} with $ g(t,x)=\Psi(x - z^{R}_{t_0}(t)) $ we get
  \begin{eqnarray}
 \frac{d}{dt}I^{+R}_{t_0}(t) & = &-\dot{z}(t) \int_{\R} \Psi' \Bigl[ u^2+u_x^2 +\gamma y \Bigr] +\frac{\gamma}{2} \int_{\R} (u^2-u_x^2) \Psi'\nonumber \\
 &  & + \int_{\R} ( u u_x^2 + \gamma y u ) \Psi'+2 \int_{\R} u h \Psi' \nonumber \\
  &\le &   - \dot{z}(t) \int_{\R} \Psi' \Bigl[ u^2+u_x^2 +\gamma y \Bigr] +\frac{\gamma}{2} \int_{\R} (u^2-u_x^2) \Psi'+J_1+J_2\, . \label{go3}
 \end{eqnarray}
 The crucial observation is that the second term in the right-hand side of \eqref{go3} that provides from the momentum part of $ I $  can be absorbed thanks to 
 the term coming from the  energy part of $ I $. More precisely, for  $ 0\le \gamma\le \frac{3}{2}(1-\alpha) c_0 $, it holds 
 $$
 -\dot{z}(t) \int_{\R} \Psi' \Bigl[ u^2+u_x^2 +\gamma y \Bigr] +\frac{\gamma}{2} \int_{\R} (u^2-u_x^2) \Psi'\le -\frac{(1-\alpha) c_0}{4} \int_{\R} \Psi' \Bigl[ u^2+u_x^2 +\gamma y \Bigr] 
 $$
 Now, the terms $J_1$ and $ J_2$ are treated as usually. 
 To estimate $ J_{1} $ we divide $ \R $ into two regions relating to the size of $ |u| $ as follows
\begin{eqnarray}
J_{1}(t) &= & \int_{|x-x(t)|<R_0} (u u_x^2+\gamma y u )\Psi'
+ \int_{|x-x(t)|>R_0} (u u_x^2+\gamma y u )\Psi'\nonumber \\
 & = & J_{11}+J_{12}\quad . \label{J0}
\end{eqnarray}
Observe that \eqref{condz} ensures that $ \dot{x}(t)-\dot{z}(t)\ge \beta c_0 $ for all $t\in \R $  and thus, for $ |x-x(t)|<R_0 $,
 \begin{equation}\label{to1}
  x-z_{t_0}^{R}(t)=x-x(t)-R+(x(t)-z(t))-(x(t_0)-z(t_0))\le  R_0-R-\beta c_0 (t_0-t) 
  \end{equation}
  and thus the decay properties of $ \Psi' $ ensure that 
\begin{eqnarray}
J_{11} (t) &\lesssim & \Bigl[\|u(t)\|_{L^\infty} (\|u_x(t)\|_{L^2}^2+c_0\|y(t)\|_{L^1})\Bigr]  e^{R_0/6}  e^{-R/ 6}
e^{-\frac{\beta}{6} c_0(t_0-t)} \nonumber \\
 & \lesssim  &  \| u_0\|_{H^1}(\|u_0\|_{H^1}^2+c_0\|y_0\|_{L^1}) e^{R_0/6} e^{-R/6}
e^{-\frac{\beta}{6} c_0(t_0-t)} \quad . \label{J11}
\end{eqnarray}
On the other hand,  \eqref{dif}  ensures that for all $ t\in [t_0-T,t_0] $ it holds
\begin{eqnarray}
J_{12} &\le & 4 \| u\|_{L^\infty(|x-x(t)|>R_0)} \int_{|x-x(t)|>R_0} (u_x^2+\gamma y)\Psi'\nonumber \\
 & \le & \frac{ (1-\alpha) c_0 }{8}  \int_{|x-x(t)|>R_0} (u_x^2+\gamma y ) \Psi' \quad .\label{J12}
\end{eqnarray}
 It thus remains to estimate
 $ J_2(t) $.
For this,  we decompose again $
  \R $ into two regions relating to the size of $ |u| $.
   First proceeding as
    in \eqref{J11} we easily check that 
  \begin{eqnarray}
& & \int_{|x-x(t)|<R_0}u \Psi'
(1-\partial_x^2)^{-1}(2u^2+u_x^2 ) \nonumber \\
& & \le 4 \|u\|_{L^\infty} \sup_{|x-x(t)|<R_0}
|\Psi'(x- z^{R}_{t_0}(t))|\int_{\R} e^{-|x|} \ast (u^2+u_x^2 ) \, dx \nonumber \\
 &  & \le C \|u_0\|_{H^1}^3 e^{R_0/6} e^{-R/6}
 e^{-\frac{\beta}{6}c_0(t-t_0)} \label{J31}
\end{eqnarray}
since
\begin{equation}
 \forall f\in L^1(\R), \quad (1-\partial_x^2)^{-1} f =\frac{1}{2} e^{-|x|} \ast f \quad .
 \label{tytu}
 \end{equation}
Now in the region $ |x-x(t)|>R_0$, noticing that $ \Psi' $ and
$ u^2+u_x^2/2 $ are non-negative, we  get
 \begin{eqnarray}
 & & \int_{|x-x(t)|>R_0} u \Psi'
(1-\partial_x^2)^{-1}(2u^2+u_x^2 ) \nonumber \\
 &  \le &
\|u(t)\|_{L^\infty(|x-x(t)|>R_0)}\int_{|x-x(t)|>R_0}\Psi'(
(1-\partial_x^2)^{-1}(2u^2+u_x^2) \nonumber \\
&  \le &  \|u(t)\|_{L^\infty(|x-x(t)|>R_0)}  \int_{\R} (2u^2+u_x^2) (1-\partial_x^2)^{-1}
\Psi'
\end{eqnarray}
On the other hand, from
  \eqref{psi3} and \eqref{tytu}   we  infer  that 
  $$
(1-\partial_x^2) \Psi' \ge \frac{1}{2} \Psi' \Rightarrow
(1-\partial_x^2)^{-1} \Psi'\le 2  \Psi' \; .
  $$
Therefore, on account of \eqref{dif},
\begin{eqnarray}
 & &  \int_{|x-x(t)|>R_0} u  \Psi'
(1-\partial_x^2)^{-1}(2u^2+u_x^2) \nonumber \\
&  &  \le 2 \|u(t)\|_{L^\infty(|x-x(t)|>R_0)} \int_{\R}
(2u^2+u_x^2)
 \Psi' \nonumber \\
 &  &  \le   \frac{ (1-\alpha)  c_0}{8}
\int_{\R} (u^2+u_x^2)
 \Psi' \label{J32}
\end{eqnarray}
Gathering \eqref{J0}, \eqref{J11}, \eqref{J12}, \eqref{J31} and
\eqref{J32} we conclude that there exists  $C $  only
depending on  $R_0 $ and $ E(u) $ 
  such
that for  $ R \ge R_0 $ and $ t\in [-T+t_0,t_0] $ it holds 
\begin{equation}
\frac{d}{dt} I^{+R}_{t_0}(t) \le  C  e^{-R/6} e^{-\frac{\beta}{6}(t_0-t)} \; .
\label{nini}
\end{equation}
Integrating between $ t$ and $ t_0$  we obtain \eqref{mono} for  any $  t \in [ t_0-T,t_0] $ and $ u $ replaced by $ u_n $ with $ n\ge n_0$. Note that the constant appearing in front of the exponential now also depends on $ \beta$. 
 The convergence results \eqref{cvH1}-\eqref{weakcv1} then ensure that \eqref{mono} holds also for  $ u $ and   $t\in[ t_0-T,t_0] $ and  the result for $ t\le t_0 $ follows since $ T>0 $ is arbitrary. Finally, \eqref{mono2} can be proven in exactly the same  way by noticing that  for  $|x-x(t)|<R_0 $ it holds 
  \begin{equation} \label{to2}
   x-z_{t_0}^{-R}(t)=x-x(t)+R+(x(t)-z(t))- (x(t_0)-z(t_0))\ge  -R_0+R+\beta c_0 (t-t_0) \; .
  \end{equation}
\end{proof}
\noindent
{\bf Proof of Proposition \ref{prodecay}}
First, since $ u $ is $Y$-almost localized, it is clear that $ u $ satisfies the hypotheses of Lemma \ref{almostdecay} for $\alpha=1/3$. 
We  fix $\alpha=1/3 $ and take $ \beta=1/3 $, $ \gamma=c_0$ and $ z(\cdot)=\frac{2}{3} x(\cdot)$ which clearly satisfy \eqref{condz}. Let us   show that $ I^{+R}_{t_0}(t) \tendsto{t\to-\infty} 0 $ which
together with \eqref{mono}
 will clearly lead to
 \begin{equation}
I^{+R}_{t_0}(t_0) \le  C e^{-R/6} \label{gigi}\quad .
 \end{equation}
For $ R_\varepsilon>0 $ to be specified later we decompose
  $ I^{+R}_{t_0} $ into
  \begin{eqnarray*}
I^{+R}_{t_0}(t)& = &\dist{u^2(t)+u_x^2(t)+c_0  y(t)}{\Psi(\cdot-z^{R}_{t_0}(t))\Bigl(1-\phi(\frac{\cdot-x(t)}{R_\varepsilon})\Bigr)}\\
& & +\dist{u^2(t)+u_x^2(t)+c_0 y(t)}{\Psi(\cdot-z^R_{t_0}(t))
\phi(\frac{\cdot-x(t)}{R_\varepsilon})}
\\
 &= & I_1(t)+I_2(t) \quad .
  \end{eqnarray*}
  where $ \phi\in C^\infty(\R) $ is supported in $[-1,1] $ with $ 0\le \phi\le 1 $ on $ [-1,1] $ and $ \phi\equiv 1 $ on $[-1/2,1/2]$. 
From the $Y$-almost localization hypothesis, for any $ \varepsilon>0
$ there exists $ R_\varepsilon>0 $
 such that $ I_1(t) \le \varepsilon/2 $. On the other hand, we observe that
 $$
 I_2(t) \le  (\|u_0\|_{H^1}^2+c_0\|y_0\|_{\mathcal{M}}) \Psi\Bigl(R_\varepsilon-R-\frac{1}{3}(x(t_0)-x(t))\Bigr) \quad.
 $$
  But
  $ \dot{x}>c_0>0 $  obviously  imply that, for $ |x-x(t)|\le R_\varepsilon$, 
  $$x-z^{+R}_{t_0}(t)=x-x(t)-R -\frac{1}{3} (x(t_0)-x(t))  \le R_\varepsilon-R-\frac{1}{3} c_0 (t_0-t) \tendsto{t\to -\infty} -\infty
  $$
  which
  proves our claim since $\displaystyle\lim_{x\to -\infty} \Psi(x) =0 $. \\
 It follows from  \eqref{gigi} that for all $ t\in\R $, all 
  $x_0>0 $  and all $ \Phi\in C(\R) $ with $0\le \Phi\le 1 $ and $ \supp \Phi \subset [x_0, +\infty[ $.
$$
 \int_{\R} (u^2(t)+u_x^2(t))  \Phi(\cdot-x(t)) \, dx + c_0 \dist{ \Phi(\cdot-x(t))}{y(t)} \le C \, \exp( -R/6) \; .
 $$
  The invariance of (C-H) under the transformation $ (t,x) \mapsto (-t,-x) $ yields the result for $  \supp \Phi \subset ]-\infty,-x_0] $ which completes the proof of the proposition.
 \section{Liouville result  for  $Y$-almost localized solution moving to the right}
  In this section we will need the following lemma (see for instance \cite{Ifti})
 \begin{lemma}\label{BV}
 Let $\mu $  be a finite nonnegative measure on $  \R$ . Then $\mu $ is the sum of a nonnegative non atomic  measure $ \nu $ and a countable sum of positive Dirac measures (the discrete part of $\mu$). Moreover, for all $ \varepsilon>0 $ there exists $\delta>0 $  such that, if $I$ is an interval of length less than 
 $\delta$ , then $ \nu(I) \le \varepsilon $.
 \end{lemma}
 \subsection{Boundedness from above of the  support  of the momentum density}
\begin{proposition} 
\label{pro3}Let $ u \in C(\R; Y_+) $ be a $Y$-almost localized solution of \eqref{CH} with $ \dot{x} \ge c_0>0$. There  exists $ r_0>0 $ such that\ for all $t\in \R$, it holds 
 \begin{equation}\label{pro3.1}
 \supp y(t,\cdot+x(t))\subset ]-\infty,r_0] ,
\end{equation}
and
\begin{equation}\label{pro3.2}
u(t,x(t)+r_0)=-u_x(t,x(t)+r_0)\ge \frac{e^{-2r_0}}{4 \sqrt{r_0}} \sqrt{E(u)} \, .
\end{equation}
 \end{proposition}
 \begin{proof}
 Clearly, it suffices to prove the result for $ t=0 $. 
  Let $ u \in C(\R; H^1) $,with $y=u-u_{xx}\in C_{w}(\R;\M_+) $,  be a $Y$-almost localized solution to \eqref{CH} and   let $ \phi\in C^\infty(\R) $ with $ \phi\equiv 0 $ on $\R_-, \; \phi'\ge 0 $ and $ \phi\equiv 1  $ on $[1,+\infty]$. We claim that there exists $ r_0>0 $ such that 
    \begin{equation}\label{claim1}
  \dist{y(0)}{\phi(\cdot-(x(0)+r_0))}=0 
  \end{equation}
  which proves the result  since $ y(0)\in \M_+ $.\\
  We prove \eqref{claim1} by contradiction. We approximate $u_0=u(0) $ by the sequence of smooth functions $ u_{0,n}=\rho_n\ast u_0 $ that belongs to $H^\infty(\R) \cap Y_+ $ so that \eqref{cont1}-\eqref{cont2} hold for any  $ T>0 $. We denote by $ u_n $ the solution to \eqref{CH} emanating from $ u_{0,n} $ and  by $ y_n=u_n-u_{n,xx}$ its momentum density. Let us recall that Proposition \ref{smoothWP} ensures that 
  $ u_n\in C(\R;H^\infty(\R)) $ and $ y_n\in C_{w}((\R;L^1(\R)) $.
We fix $ T>0 $ and we take $ n_0\in \N $ large enough so that for all $ n\ge n_0 $, 
\begin{equation}\label{approxu}
\|u_n -u \|_{L^\infty(]-T,T[; H^1)} <  \frac{1}{10}\min(c_0, \|u(0)\|_{H^1})
\end{equation}
and 
\begin{equation}\label{approxy0}
\|y_{0,n} -y_0 \|_{\mathcal{M}} < \frac{\varepsilon_0}{2} \quad .
\end{equation}
where $ \varepsilon_0>0 $ will be specified later.
   Thanks to the $Y$-almost localization of $ u $ , there exists $ r_0>0 $ such that 
   \begin{equation}\label{hu}
  \|u(t)\|_{H^1(\R/]x(t)-r_0,x(t)+r_0[)} \le  \frac{1}{10}\min(c_0, \|u(0)\|_{H^1}), \forall t\in \R \, .
  \end{equation}
  Note that by Sobolev's embedding theorem, it also holds 
   \begin{equation}\label{hu1}
  u(t,x(t)+x) \le  \frac{1}{10}\min(c_0, \|u(0)\|_{H^1}), \forall (|x|,t)\in [r_0,+\infty[\times\R \, .
  \end{equation}
  Combining these two estimates with \eqref{approxu} we infer that for $ n\ge n_0 $, 
   \begin{equation}\label{hun}
  \|u_n(t)\|_{H^1(\R/]x(t)-r_0,x(t)+r_0[)} \le  \frac{1}{5}\min(c_0, \|u(0)\|_{H^1}), \forall t\in [-T,T] \, 
  \end{equation}
  and
  \begin{equation}\label{hu1n}
  u_n(t,x(t)+x) \le  \frac{1}{5}\min(c_0, \|u(0)\|_{H^1}), \forall (|x|,t)\in [r_0,+\infty[\times [-T,T] \, .
  \end{equation}
 
 Now, we introduce the flow $ q_n$ associated with $ u_n $ defined by 
 \begin{equation}\label{defqn}
  \left\{ 
  \begin{array}{rcll}
  q_{n,t}(t,x) & = & u_n(t,q_n(t,x))\, &, \; (t,x)\in \R^2\\
  q_n(0,x) & =& x\, & , \; x\in\R \; 
  \end{array}
  \right. .
  \end{equation} 
  Following \cite{C}, we know that for any $ t\in \R $,
  \begin{equation}\label{yy}
  y_n(0,x)=y_n(t,q_n(t,x)) q_{n,x}(t,x)^2 
  \end{equation}
    Indeed, on one hand,  \eqref{CHy} clearly ensures that 
 $$
 \frac{\partial}{\partial t} \Bigl( y_n(t,q_n(t,x))e^{2\int_0^t u_{n,x}(s, q_n(s,x))\, ds} \Bigr) =0 
$$ 
and, on the other hand, \eqref{defqn}  ensures that $ q_{n,x}(0,x)=1 $, $\forall x\in \R $, and 
\begin{equation}\label{formula1}
\frac{\partial}{\partial t} q_{n,x}(t,x)=u_x(t,q(t,x)) q_x(t,x) \Rightarrow q_{n,x}(t,x) =\exp\Bigl( \int_0^t u_x(s,q(s,x))\,ds\Bigr) \; .
\end{equation}
  We claim that for all $ n\ge n_0 $ and $ t\in [-T,0] $ , 
   \begin{equation}\label{claim2}
 q_n(t,x(0)+r_0) -x(t) \ge r_0+\frac{c_0}{2} |t|  \; .
  \end{equation}
  Indeed, fixing $ n\ge n_0 $, in view of  \eqref{hu1n} and the continuity of $ u_n $ there exists $ t_0\in [-T,0[$ such that for all $t\in [t_0,0] $, 
  $$
  u_n(t,q_n(t,x(0)+r_0))\le \frac{c_0}{4}
  $$
  and thus according to \eqref{defqn}, for all $ t\in [t_0,0] $, 
  $$
  \frac{d}{dt} q_n(t,x(0)+r_0) \le \frac{c_0}{4} 
  $$
  which leads to 
  $$
  q_n(t,x(0)+r_0)-x(t) \ge r_0+\frac{c_0}{2} |t|, \quad  t\in [t_0,0]\; .
  $$
  This proves \eqref{claim2} by a continuity argument.
   We thus deduce from Proposition \ref{prodecay} that for all $ t\in [-T,0]  $ and all $ x\ge 0 $,
  \begin{equation}
  u(t,q_n(t,x(0)+r_0+x)  \le C \exp \Bigl( -\frac{1}{6}(r_0+ c_0 |t|/2 )\Bigr)\\
  \end{equation}
  Therefore, in view of \eqref{cont1} and \eqref{dodo}, there exists $n_1\ge n_0 $ such that for all $ t\in [-T,0] $ and all $ x\ge 0 $,
  \begin{equation}
  u_n(t,q_n(t,x(0)+r_0+x) +|u_{n,x}(t,q_n(t,x(0)+r_0+x)| \le 4 C \exp \Bigl( -\frac{1}{6}(r_0+ c_0 |t|/2 )\Bigr)\\
  \end{equation}
  The formula (see \eqref{formula1})
  \begin{equation}\label{formula}
  q_{n,x}(t,x)=\exp \Bigl( -\int_t^0 u_{n,x}(s,q_n(s,x))\, ds \Bigr) 
  \end{equation}
   thus ensures that $ \forall t\in [-T,0] , \; \forall x\ge 0 $ and $ \forall n\ge n_0 $, 
  $$
   \exp\Bigl( -4C \int_{-T}^0 e^{ -\frac{1}{6}(r_0+ c_0 |s|/2)}\, ds \Bigr) \le  q_{n,x}(t,x(0)+r_0+x) \le \exp\Bigl( 4C \int_{-T}^0 e^{ -\frac{1}{6}(r_0+ c_0 |s|/2)}\, ds \Bigr)
  $$
  Setting $ C_0:= e^{\frac{48 C e^{-r_0/6}}{c_0}} $ this leads to 
  \begin{equation}\label{to}
  \frac{1}{C_0} \le  q_{n,x}(t,x(0)+r_0+x) \le C_0 \, , \; \forall t\in [-T,0]\; .
  \end{equation}
  Now, if   \eqref{claim1} would  not be  true then, in view of \eqref{approxy0} there exists $ r_0'>r_0 $ and $ \varepsilon_0>0 $  such that $ \forall n\ge n_1$, 
  $$
  \int_{x(0)+r_0}^{x(0)+r_0'} y_n(0,x)\, dx\ge \varepsilon_0>0 \; .
  $$
It then follows from \eqref{yy} that for all $ t\in [-T,0] $,
$$
\int_{x(0)+r_0}^{x(0)+r_0'} y_n(t,q_n(t,x)) q_{n,x}(t,x)^2\, dx \ge \varepsilon_0
$$ 
and \eqref{to} leads to 
$$
\int_{x(0)+r_0}^{x(0)+r_0'} y_n(t,q_n(t,x)) q_{n,x}(t,x)\, dx \ge C_0^{-1} \varepsilon_0
$$ 
Therefore, the change of variables $ z=q_n(t,x) $ yields 
$$
\int_{q_n(t,x(0)+r_0)}^{q_n(t,x(0)+r_0')} y_n(t,z)\, dz \ge  C_0^{-1}  \varepsilon_0
$$
and \eqref{claim2} ensures that 
\begin{equation}\label{to8}
\int_{x(t)+r_0+c_0 |t|/2}^{+\infty} y_n(t,z)\, dz \ge   C_0^{-1}  \varepsilon_0\, , \; \forall t\in [-T,0]  \; .
\end{equation}
Letting $ n\to +\infty $ using \eqref{cont2} and then letting $ T \to \infty $,  this ensures that 
$$
\dist{y(t)}{ \phi(\cdot-x(t)-r_0-c_0 |t|/2)}\ge    C_0^{-1}  \varepsilon_0\, , \; \forall t\le 0  \; .
$$
This  clearly contradicts the $Y$-almost localization of $ u $ and thus completes the proof of  \eqref{pro3.1}.

Let us now prove \eqref{pro3.2}.  We first notice that  thanks to \eqref{hu} and  the conservation of the $ H^1$-norm  it holds 
$$
\|u(t,\cdot-x(t))\|_{H^1(]-r_0,r_0[)} \ge \frac{1}{2}\sqrt{E(u)}\; .
$$
   But since for all $ t\in \R $, $ |u_x(t)|\le u(t) $ on $ \R $, this forces  
  \begin{equation}\label{hg}
   \max_{[-r_0,r_0]} u^2(t, \cdot-x(t)) \ge \frac{1}{2r_0} \|u(t, \cdot -x(t))\|_{L^2(]-r_0,r_0[)}^2 \ge \frac{1}{8 r_0}E(u)\; .
  \end{equation}
  Moreover, since $ u_x\ge -u $ on $\R^2$, for any $ (t,x_0)\in\R^2 $ it holds
 $$
 u(t,x) \le u(t,x_0) e^{-x+x_0}, \quad \forall x\le x_0 \; .
 $$
Applying this estimate with  $ x_0=x(t)+r_0 $ we obtain that 
 $$
 u(t,x(t)+r_0) \ge  \max_{[-r_0,r_0]} u(t, \cdot-x(t))  e^{-2r_0}
 $$
which, combined with \eqref{hg} yields \eqref{pro3.2}.
 \end{proof}
 
\subsection{Study of the first  jump  of $ u_x $}
We define 
$$
x_+(t)=\inf \{ x\in\R, \, \supp y(t)\subset  ]-\infty, x(t)+x] \} 
$$
In the sequel we set 
$$
\alpha_0:=\frac{e^{-2r_0}}{4 \sqrt{r_0}}\sqrt{E(u)}
$$
to simplify the expressions. According to Proposition \ref{pro3}, $  t\mapsto x_+(t) $ is well defined with values in $ ]-\infty, r_0] $ and 
\begin{equation}
u(t,x(t)+x_+(t))=-u_x(t,x(t)+x_+(t))\ge \alpha_0 \, .
\end{equation}
The following lemma ensures that $t\mapsto x(t)+x_+(t)$ is an integral line of $ u $. 
\begin{lemma} \label{lemmaq}
For all $ t\in \R$, it holds
\begin{equation} \label{qq}
x(t)+x_+(t)=q(t,x(0)+x_+(0)) \; .
\end{equation}
where $ q(\cdot,\cdot)$ is defined by 
 \begin{equation}\label{defq}
  \left\{ 
  \begin{array}{rcll}
  q_t (t,x) & = & u(t,q(t,x))\, &, \; (t,x)\in \R^2\\
  q(0,x) & =& x\, & , \; x\in\R \; 
  \end{array}
  \right. .
  \end{equation} 
\end{lemma}
\begin{proof} 
First, it is worth noticing that the result would follow directly from the definition of $ x_+(\cdot)$ and $ q(\cdot) $ if  $ u$ would belong to $ C(\R; H^3(\R)) $. Since \eqref{CH} is reversible with time, it suffices to prove  that  \eqref{qq}  holds for all $ t\in [0,1]$. We proceed by contradiction by  assuming that there exits $t_0\in ]0,1]$
 such that $ q(t_0, x(0)+x_+(0) =x(t_0)+x_+(t_0) +\lambda_0 $ with $ \lambda_0\neq 0$. We separate the cases $ \lambda_0>0 $ and $ \lambda_0<0 $. 
 
 In the case 
  $ \lambda_0<0 $,  by the continuity and   monotonicity of $ x\mapsto q(t,x) $, $ t\in \R$, there exists $ \delta_0>0 $ such that 
   $q(t_0,x(0)+x_+(0)+\delta_0) <x(t_0)+x_+(t_0) -\frac{|\lambda_0|}{2}$.  
   We approximate  $ u$ at time $ t=0$ by the sequence $ \varphi_n= \rho_{n} \ast u(0) $ where $ \rho_n $ is defined in \eqref{rho}. 
By construction $ \varphi_n\in H^\infty(\R) \cap Y_+ $ and $ (1-\partial_x^2) \varphi_{n}\equiv 0 $ on $[x(0)+x_+(0)+\delta_0,+\infty[ $ as soon as 
   $ n>2/\delta_0$.  Let $ u_n $ be the solution of  \eqref{CH} such that $ u_n (t_0)=\varphi_n$.  
   Defining $ q_n : \R\to \R $ by 
   $$
    \left\{ 
  \begin{array}{rcll}
  \frac{d}{dt}q_n(t) & = & u_n(t,q_n (t))\, &, \; \forall t\in \R \\
  q_n(0) & =& x(t_0)+x_+(t_0)+\delta_0\, &  \; 
  \end{array}
  \right. 
$$
   it follows from \eqref{yy} that $ (1-\partial_x^2) u_n(0,\cdot)\equiv 0 $ on $ [q_n(t_0),+\infty[ $ for  $n>2/\delta_0$.
 On account of \eqref{cont1}, 
   $ q_n(\cdot) \to q(\cdot,x(t_0)+x_+(t_0)+\delta_0) $ in $C([0,1]) $ and thus $  (1-\partial_x^2) u_n(t_0)\equiv 0 $ on 
    $[x(t_0)+x_+(t_0)-\frac{|\lambda_0|}{4},+\infty[ $ for $ n $ large enough. \eqref{cont2} then  ensures that $ \supp y(t_0) \subset ]-\infty,    x(t_0)+x_+(t_0)-\frac{|\lambda_0|}{8}[$
     which contradicts the definiton of $x_+(t_0) $.
     
In the case $ \lambda_0>0 $,   there exists $ \delta_0>0 $ such that  $q(t_0,x(0)+x_+(0)-\delta_0) >x(t_0)+x_+(t_0) +\frac{|\lambda_0|}{2}$.  
     Then we obtain again a contradiction by following the same arguments, exchanging the role of $ t=0$ and $ t=t_0$. This completes the proof of the lemma.
\end{proof}
 In the sequel we define $ q^* \, :\, \R \to \R $ by 
 \begin{equation}\label{defq*}
 q^*(t)=q(t,x(0)+x_+(0))=x(t)+x_+(t), \quad \forall t\in \R \; .
 \end{equation}

\begin{proposition}\label{propa}
Let  $ a \, :\, \R\to \R $ be the function defined by 
\begin{equation}\label{defa}
a(t)=u_x(t,q^*(t)-)-u_x(t,q^*(t)+), \quad \forall t\in\R \, .
\end{equation}
Then $ a(\cdot) $ is  a bounded non decreasing differentiable   function  on $ \R $ with values in $[\frac{\alpha_0}{8}, 2\sqrt{E(u)}]$  such that 
\begin{equation}\label{esta}
a'(t)=\frac{1}{2}(u^2-u_x^2)(t,q^*(t)-) , \;\forall t\in \R .
\end{equation}
\end{proposition}
\begin{proof}
First, the fact that $ a(t) \le 2 \sqrt{E(u)}$ follows directly from $ |u_x|\le u \le \|u\|_{H^1} $. 
To prove that $ a(t) \ge \frac{\alpha_0}{8}$, 
 we proceed by contradiction. So let us assume that there exists $ t_0\in\R $ such that  $ a(t_0)<\alpha_0/8 $. Since 
  $ y(t_0)\in \M_+ $ with $ \supp y(t_0)\subset ]-\infty, q^*(t_0)]$, according to Lemma \ref{BV} we must have
$$
\lim_{z\nearrow q^*(t_0)} \|y(t_0)\|_{{\mathcal M}(]z,+\infty[)}<\frac{\alpha_0}{8} \;.
$$
Without loss of generality we can assume that $ t_0=0 $ and thus  there exists $ \beta_0>0 $ such that 
\begin{equation}\label{toto1}
 \|y(0)\|_{{\mathcal M}(]q^*(0))-\beta_0,+\infty[)}<\frac{\alpha_0}{8} \; .
\end{equation}
By convoluting $ u_0$ by $ \rho_n $ (see \eqref{rho}), for some $ n\ge 0 $, we can  approach $ u_0 $ by a smooth function $\tilde{u}_0 \in Y_+\cap H^\infty(\R) $.  
 Taking $ n $ large enough, we may  assume that there exists  $ \tilde{x}_+>x_+(0)$ close to $ x_+(0) $,  such that 
  \begin{equation}\label{hu21}
 \tilde{y}_0=(1-\partial_x^2) \tilde{u}_0 \equiv 0 \mbox{ on }[x(0)+\tilde{x}_{+},+\infty[
 \end{equation}
 and
  \begin{equation}\label{hu2}
  \|\tilde{y}_0\|_{L^1(]x(0)+\tilde{x}_+-\beta_0,+\infty[)}\le  \frac{\alpha_0}{8}+ \frac{\alpha_0}{2^6}\;, 
 \end{equation}
 where $\tilde{y}_0=\tilde{u}_0-\tilde{u}_{0,xx} $. 
Moreover, defining $ \tilde{q}_2 \, :\, \R\to \R $ by 
$$
\tilde{q}_2(t)=\tilde{q}(t,x(0)+\tilde{x}_+)
$$
where $ \tilde{q}(\cdot,\cdot) $ is defined by   \eqref{defq} with $ u $ replaced by $ \tilde{u}$, 
   \eqref{cont1} enables us  to assume that   the emanating solution $ \tilde{u } $ satisfies 
 \begin{equation}\label{appo}
 \|u(t) - \tilde{u}(t)\|_{H^1} \le \frac{\alpha_0}{2^6}
 \end{equation}
 and 
 \begin{equation}\label{appo2}
 |q^*(t)-\tilde{q}_2(t)|<\frac{\alpha_0}{2^6 \sqrt{E(u)}}
 \end{equation}
 for all $ t\in [-t_1,t_1] $ with $ t_1 >0 $ to specified later. 
  It is worth noticing that \eqref{appo}-\eqref{appo2}, \eqref{hu21}, \eqref{yy}, \eqref{pro3.2} and the mean-value theorem  - recall that $ |u_x|\le u \le \sqrt{E(u)} $ - then ensure that 
 \begin{equation}\label{appo3}
-\tilde{u}_x(t, \tilde{q}_2(t)) = \tilde{u}(t, \tilde{q}_2(t)) \ge (1-2^{-5}) \alpha_0\; \quad \forall t\in [-t_1,t_1] \; .
 \end{equation}

We claim that  for all $ t\in [-t_1,0] $ it holds 
\begin{equation}\label{claim3}
\tilde{u}_x(t,x)\le -\frac{3\alpha_0}{4}\quad \mbox{on}\quad  [\tilde{q}_1(t), \tilde{q}_2(t)] \, , 
\end{equation}
where $ \tilde{q}_1(t) $ is defined by  $ \tilde{q}_1(t)=\tilde{q}(t,x(0)+\tilde{x}_+-\beta_0)$.\\
To see this, for $ \gamma>0 $, we set 
$$
A_{\gamma}=\{t\in \R_-\,  /\, \forall \tau\in [t,0], \; u_x(\tau,x) 	< -\gamma \; \mbox{on} \; [\tilde{q}_1(\tau), \tilde{q}_2(\tau)] \, \} \; .
$$
Recalling \eqref{pro3.2}, \eqref{hu2}, \eqref{appo3} and  that $ \tilde{u} \ge 0 $ , we get for $ 0\le \beta\le \beta_0$, 
\begin{eqnarray*}
\tilde{u}_x(0,x(0)+\tilde{x}_{+}-\beta) &  \le &  \tilde{u}_x(0,x(0)+\tilde{x}_{+}) +  \|\tilde{y}_0\|_{L^1(]x(0)+\tilde{x}_{+}-\beta_0,+\infty[)} \\
 &\le & -\alpha_0+\frac{\alpha_0}{2^5}+\frac{\alpha_0}{8} +\frac{\alpha_0}{2^5}< -\frac{3\alpha_0}{4}\; ,
\end{eqnarray*}
which ensures that $  A_\frac{3\alpha_0}{4} $ is non empty. By  a continuity argument, it thus suffices to prove that $  A_\frac{\alpha_0}{2} \subset    A_\frac{3\alpha_0}{4} $.
First we notice that for any $ t\in  A_\frac{\alpha_0}{2}  $ and any $ x\in [\tilde{q}_1(t), \tilde{q}_2(t)]$, the definition of  $ A_\frac{\alpha_0}{2}  $ ensures that 
$$
\tilde{q}_x(t,x)=\exp\Bigl( -\int_t^0 \tilde{u}_x(\tau,\tilde{q} (\tau,x))\, d\tau \Bigr)\ge 1 \; ,
$$
where $ \tilde{q}(\cdot,\cdot) $ is the flow associated to $ \tilde{u}$  by \eqref{defq}.
Therefore, $ \tilde{u}\ge 0 $, $ \tilde{y}\ge 0 $, a  change of variables, \eqref{yy}  and \eqref{hu2} ensure that for any $ x\in [\tilde{q}_1(t), \tilde{q}_2(t)]$,
$$
\int_{x}^{\tilde{q}_2(t)} \tilde{u}_{xx} (t,s) \, ds\ge - \int_{x}^{\tilde{q}_2(t)}\tilde{y} (t,s) \, ds\ge - \int_{\tilde{q}_1(t)}^{\tilde{q}_2(t)} \tilde{y} (t,s) \, ds=-\int_{x(0)+\tilde{x}_+ -\beta_0}^{x(0)+\tilde{x}_+} \tilde{y}(t,\tilde{q}(t,s))\tilde{q}_x(t,s) \, ds
$$
$$
\ge  -\int_{x(0)+\tilde{x}_+ -\beta_0}^{x(0)+\tilde{x}_+} \tilde{y}(t,\tilde{q}(t,s))\tilde{q}_x(t,s)^2 \, ds=   -\int_{x(0)+\tilde{x}_+ -\beta_0}^{x(0)+\tilde{x}_+} \tilde{y}_0(s) \, ds\ge -\frac{\alpha_0}{8}-\frac{\alpha_0}{2^6} 
$$
and \eqref{appo3}  yields
$$
\tilde{u}_x(t,x) = \tilde{u}_x(t,\tilde{q}_2(t))-\int_{x}^{\tilde{q}_2(t)} \tilde{u}_{xx} (t,s) \, ds \le -\alpha_0+\frac{\alpha_0}{8}+\frac{\alpha_0}{2^4}<-\frac{3\alpha_0}{4} \; ,
$$
which proves the desired result.

 We deduce from \eqref{claim3} that $ \forall t\in [-t_1, 0] $,
\begin{eqnarray*}
\frac{d}{dt} (\tilde{q}_2(t)-\tilde{q}_1(t)) & = & \tilde{u}(\tilde{q}_2(t))-\tilde{u}(\tilde{q}_1(t)) \\
& =& \int_{\tilde{q}_1(t)}^{\tilde{q}_2(t)} \tilde{u}_x(t,s) \, ds \\
& \le & -\frac{\alpha_0}{2} (\tilde{q}_2(t)-\tilde{q}_1(t)) \; .
\end{eqnarray*}
Therefore,
$$
(\tilde{q}_2-\tilde{q}_1)(t)\ge (\tilde{q}_2-\tilde{q}_1)(0) e^{-\frac{\alpha_0}{2} t} = \beta e^{-\frac{\alpha_0}{2} t}  \; .
$$
 On the other hand,  since  according to \eqref{appo3} and \eqref{claim3}, 
$ \tilde{u}(t, \tilde{q}_2(t))\ge 2 \alpha_0/3 $ and $\tilde{u}_x \le 0 $ on $]\tilde{q}_1(t), \tilde{q}_2(t)[ $, we deduce that   
$$
\tilde{u}(t,\tilde{q}_1(t))\ge \tilde{u}(t,\tilde{q}_2(t))\ge 2\alpha_0/3 , \quad \quad \mbox{ on } [-t_1,0] \; .
$$
Coming back to the solution $ u $ emanating from $ u_0 $, it follows from  \eqref{appo}  that 
$$
\min\Bigr(u(t,\tilde{q}_1(t_1)),u(t,\tilde{q}_2(t_1))\Bigr)\ge \frac{\alpha_0}{2} \mbox{ with } (\tilde{q}_2-\tilde{q}_1)(t_1)\ge \beta e^{-\frac{\alpha_0}{2} t} \; , \quad
\forall t\in [-t_1,0]\, .
$$
Taking $ t_1>0 $ large enough, this  contradicts the $Y$-almost localization of $ u$ which proves that $ a(t)\ge \frac{\alpha_0}{8}$ and thus $ u_x(t,\cdot) $ has got a jump at 
 $ x(t)+x_+(t) $.
 
It is worth noticing that,according Lemma \ref{BV}, this ensures that for all $ t\in\R $, one can decompose $ y(t)$ as 
\begin{equation}\label{decompositiony}
y(t)= \nu(t)+ a(t) \delta_{x(t)+x_+(t)}+\sum_{i=1}^\infty a_i(t) \delta_{x_i(t)}
\end{equation}
where $ \nu(t) $ is a non negative non atomic measure with $ \nu(t)\equiv 0 $ on $  ]x(t)+x_+(t), +\infty[ $,  $\{a_i\}_{n\ge 1 } \subset (\R_+)^{\N} $  with $ \sum_{i=1}^\infty a_i(t)<\infty $ and $ x_i(t)<x(t)+x_+(t) $ for all $ i\in \N^* $.

It remains to prove that for all pair of real numbers $(t_1,t_2) $ with  $ t_1<t_2 $,
\begin{equation}\label{tt}
a(t_2)-a(t_1)=\frac{1}{2}\int_{t_1}^{t_2} (u^2-u_x^2)(\tau,q^*(\tau)-) \, d\tau \; .
\end{equation}
Indeed, since $ |u_x|\le u $ and  $u\in C(\R^2) $ this will force $ a$ to be non decreasing continuous function on $ \R $. Then noticing that 
\begin{align*}
(u^2-u_x^2)(t,q^*(t)-)  &=  a(t)(u-u_x)(t,q^*(t)-) \\
 & =a(t)\Bigl(2u(t,q^*(t)-)-a(t)\Bigr) =a(t)\Bigl(2u(t,q^*(t))-a(t)\Bigr)
 \end{align*}
 with  $t\mapsto u(t,q^*(t))\in C(\R) $, the fundamental theorem of calculus will ensure that $ a $ is differentiable on $ \R$. 
 
 Let $ \phi \, :\, \R\to \R_+$ be a non decreasing  $C^\infty $-function such that $ \supp \phi\subset [-1,+\infty[ $ and $ \phi\equiv 1 $ on $\R_+ $. We set $ \phi_\varepsilon=\phi(\frac{\cdot}{\varepsilon}) $. Since $ u $ is continuous and $y(t, \cdot) =0 $ on $ ]x(t)+x_+(t),+\infty[ $ it follows from \eqref{decompositiony} that for all $ t\in \R $,
 $$
 a(t)=\lim_{\varepsilon\searrow 0} \langle y(t),\phi_\varepsilon (\cdot -q^*(t))\rangle \; .
 $$
 Without loss of generality, it suffices to prove \eqref{tt} for $ t_1=0 $ and $ t_2=t\in ]0,1[$.  Let $ \beta>0 $  be fixed, we claim that there exists $ \varepsilon_0>0 $ such that for all $ 0<\varepsilon<\varepsilon_0 $, 
  \begin{equation} \label{claimb}
\Bigl|   \langle y(t),\phi_\varepsilon (\cdot -q^*(t))\rangle-\langle y(0),\phi_\varepsilon (\cdot -q^*(0))\rangle-\frac{1}{2}\int_0^t \int_{\R}  (u^2-u_x^2)(\tau,q^*(t)+\varepsilon z ) 
\phi'(z)\, dz \, d\tau \Bigr|
\le  \beta , \quad \forall t\in ]0,1[
  \end{equation}
  Passing to the limit as $ \varepsilon $ tends to $ 0 $, this  leads to the desired result. Indeed, since $(u^2-u_x^2)(\tau,\cdot) \in BV(\R) $ and 
   $ \phi'\equiv 0 $ on $\R_+$, for any fixed $ (\tau,z) $,  it is clear that 
  $$
  (u^2-u_x^2)(\tau,q^*(\tau)+\varepsilon z ) \phi'(z) \tendsto{\varepsilon\to 0} (u^2-u_x^2)(\tau, q^*(\tau)-) \phi'(z)
  $$
  and,  since it is dominated by $ 2 \|u_0\|_{H^1}^2 \phi' $, the dominated convergence theorem leads to 
  \begin{eqnarray*}
  \int_0^t \int_{\R}  (u^2-u_x^2)(\tau,q^*(t)+\varepsilon z ) 
\phi'(z))\, dz \, d\tau & \tendsto{\varepsilon \to 0}  & \int_0^t \int_{\R}  (u^2-u_x^2)(\tau,q^*(\tau)- ) 
\phi'(z))\, dz \, d\tau \\
&  &= \int_{0}^{t} (u^2-u_x^2)(\tau,q^*(\tau)-) \, d\tau \; .
  \end{eqnarray*}
 To prove \eqref{claimb} we first notice that  according to \eqref{decompositiony} for any $ \alpha>0 $ there exists $ \gamma(\alpha) >0 $ such that 
  \begin{equation}\label{uu}
 \|y(0)\|_{{\mathcal M}(]q^*(0)-\gamma(\alpha),q^*(0)[)}<\alpha \; .
 \end{equation}
 We take $ \varepsilon_0= \gamma(\frac{\beta}{2} e^{-2\|u_0\|_{H^1}})$.  As above, we approximate again $ u(0) $ by a sequence $ \{u_{0,n}\} \subset H^{\infty}(\R)\cap Y_+ $ such that $ \|u_{0,n}\|_{H^1} \le 2 \|u_0\|_{H^1}$ and 
 \begin{equation}\label{approx}
 \|y(0)-y_{0,n}\|_{{\mathcal M}(\R)} \le \beta/4\; .
 \end{equation}
 where $ y_{0,n}=u_{0,n}-\partial_x^2 u_{0,n} $. We again denote respectively  by $ u_n $ and $ y_n $,  the solution to \eqref{CH} emanating from $ u_{0,n} $ and its momentum density $ u_n-u_{n,xx}$.  Let now $ q_n^* \; :\R\to \R $ be  the integral line of $ u_n $ defined by 
  $q_n^*(t)=q_n(t,q^*(0)) $ where $ q_n $ is defined in \eqref{defqn}.
 On account of \eqref{CHy}, it holds
\begin{eqnarray}
\frac{d}{dt} \int_{\R} y_n \phi_\varepsilon (\cdot -q_n^*(t)) &= &-u_n(t,q_n^*(t)) \int_{\R} y_n \phi_\varepsilon'  -\int_{\R} \partial_x (y_n u_n) \phi_{\varepsilon} - \int_{\R} y _n u_{n,x} \phi_{\varepsilon}\nonumber \\
 & = & \int_{\R} \Bigl[ u_n(t,\cdot)-u_n(t,q^*(t))\Bigr] y_n(t,\cdot) \phi_{\varepsilon}'+\frac{1}{2} \int_{\R} (u_n^2(t,\cdot)-u_{n,x}^2(t,\cdot)) \phi_{\varepsilon}' \nonumber\\
  & = &\frac{1}{\varepsilon} \int_{\R} \Bigl[ u_n(t,\cdot)-u_n(t,q^*(t))\Bigr] y_n(t,\cdot) \phi'\Bigl(\frac{\cdot -q_n^*(t)}{\varepsilon}\Bigr)\nonumber \\
   & &+\frac{1}{2} \int_{\R} (u_n^2-u_{n,x}^2)(t,q_n^*(t)+\varepsilon z) \phi'(z) \, dz \nonumber\\
 & = & I_t^{\varepsilon,n}+ II_t^{\varepsilon,n}\; .\label{fd}
\end{eqnarray} 
Since, according to \eqref{dodo}, $ |u_{n,x}|\le \|u_{0,n}\|_{H^1} \le 2   \|u_{0}\|_{H^1} $, 
\begin{eqnarray*}
|I_t^{\varepsilon,n}| & \le &  \frac{2 \|u_0\|_{H^1}}{\varepsilon} \int_{\R} |x-q_n^*(t) | y_n(t,x) \phi'(\frac{x-q_n^*(t)}{\varepsilon}) \, dx  \\
& \le & 2 \|u_0\|_{H^1} \int_{\R}  y_n(t,x) \phi'(\frac{x-q_n^*(t)}{\varepsilon}) \, dx
\end{eqnarray*}
Now,  in view of \eqref{formula} we easily get 
\begin{equation}\label{gg}
e^{-2\|u_0\|_{H^1}}\le q_{n,x}(t,z) \le e^{2\|u_0\|_{H^1}} , \quad \forall (t,z) \in ]-1,1[\times \R \, 
\end{equation}
and  the change of variables $ x=q_n(t,z) $ together with  the identity \eqref{yy}  lead to 
\begin{eqnarray*}
 \int_{\R}  y_n(t,x) \phi'(\frac{x-q_n^*(t)}{\varepsilon}) \, dx & = &  \int_{\R} y_n(t,q_n(t,z)) q_n'(t,z) \phi' (\frac{q_n(t,z)-q_n^*(t)}{\varepsilon}) \, dx \\
 & \le & e^{2\|u_0\|_{H^1}}  \int_{\R} y_n(t,q_n(t,z)) (q_n'(t,z))^2 \phi ' (\frac{q_n(t,z)-q_n^*(t)}{\varepsilon}) \, dz \\
 & \le & e^{2\|u_0\|_{H^1}} \int_{\R}y_n(0,z) \phi' (\frac{q_n(t,z)-q_n^*(t)}{\varepsilon}) \, dz\; .
\end{eqnarray*}
 But, making use of the mean value theorem,  \eqref{gg} and the definition of $ \phi $,  we obtain that, for any $ t \in [0,1]$,  $ z\mapsto  
 \phi ' (\frac{q_n(t,z)-q_n^*(t)}{\varepsilon}) $ is  supported in an interval of length at most $ \varepsilon e^{2\|u_0\|_{H^1}} $. Therefore,  according to \eqref{uu} and  \eqref{approx}, setting $  \varepsilon_0=e^{-2\|u_0\|_{H^1}} \gamma(\frac{\beta}{2} e^{-2\|u_0\|_{H^1}})$, it follows that for all $ 0<\varepsilon<\varepsilon_0 $ and all $ n\in \N $, 
 \begin{equation}\label{ese1}
 \int_0^t |I_\tau^{\varepsilon,n} |d\tau \le 3\beta/4 \; .
 \end{equation}
 To estimate the contribution of $ II_t^{\varepsilon,n}$ we first notice that thanks \eqref{cont1} it holds 
 $$
 u_{n,x} \to u_x \mbox{ in } C([-1,1]; L^2(\R))
 $$
 and for all $ t\in [-1,1]$,  Helly's theorem ensures that
 $$
u_{n,x}(t,\cdot) \to u_x(t,\cdot) \; \mbox{ a.e. on } \R \; .
$$
 Hence, for any fixed $ t\in [-1,1] $ there exists a set $ \Omega_t \subset \R $ of  Lebesgue measure zero such that $ u_x(t) $ is continuous at every point $x\in \R/\Omega_t $ and 
 $$
 u_{n,x}(t,x)\to u_x(t,x) \;, \quad \forall x\in \R/\Omega_t \; .
 $$
 Since $q_n^*(t) \to q^*(t) $, it follows that 
  $$
 u_{n,x}(t,q_n^*(t)+x)\to u_x(t,q^*(t)+x) \;, \quad \forall x\in \R/\tau_{q^*(t)}(\Omega_t) \; .
 $$
 where for any set $ \Lambda\subset \R $ and any $ a\in \R $, $ \tau_a(\Lambda)=\{x-a, a\in \Lambda\}$.\\
 Since the integrand in $ II_t^{\varepsilon,n}$ is bounded by $ 2\|u_0\|_{H^1} \phi'\in L^1(\R) $, it follows from  Lebesgue dominated convergence theorem that for any $ t\in [-1,1] $, 
 $$
  II_t^{\varepsilon,n} \tendsto{n\to \infty} \frac{1}{2}\int_{\R} (u^2-u_x^2)(t, q^*(t)+\varepsilon z) \phi'(z) \, dz \;.
  $$
  Therefore, invoking again Lebesgue dominated convergence theorem, but on $ ]0,t[ $, keeping in mind that $\{|u_n|\} | $ and $ \{|u_{n,x}|\} $ are uniformly bounded on $ \R^2 $ by $ 2\|u_0\|_{H^1} $, we finally deduce that for any fixed $ t\in ]0,1[ $, 
  \begin{equation}\label{ese2}
  \int_0^t II_{\tau}^{\varepsilon,n} \, d\tau \tendsto{n\to \infty} \frac{1}{2}\int_0^t \int_{\R}  (u^2-u_x^2)(\tau,q^*(t)+\varepsilon z ) 
\phi'(z))\, dz \, d\tau 
  \end{equation}
 Now, we fix $ t \in ]0,1[ $ and $ \varepsilon \in ]0,\varepsilon_0[ $. According to the convergence result \eqref{cont2}, for $ n $ large enough it holds 
  $$
  | \langle y_n(t)-y(t),\phi_\varepsilon (\cdot -q_*(t))\rangle|+|\langle y_n(0-y(0),\phi_\varepsilon (\cdot -q_*(0))\rangle| \le \beta/4 \; .
  $$
 which together with \eqref{fd} and \eqref{ese1}-\eqref{ese2} prove the claim \eqref{claimb}.
\end{proof}
\begin{lemma} \label{lem43}
There exists $(a_-,a_+)\in [\frac{\alpha_0}{8}, 2\|u_0\|_{H^1}]^2 $, with $ a_-\le a_+ $ such that 
\begin{eqnarray}
\lim_{t\to  +\infty} u(t,x(t)+x_+(t)) & =& \lim_{t\to +\infty} a(t)/2=a_{+}/2\; ,\\
\lim_{t\to  -\infty} u(t,x(t)+x_+(t)) & = & \lim_{t\to -\infty} a(t)/2=a_{-}/2\; ,
\end{eqnarray}
 \end{lemma}
\begin{proof} The existence of the limits at $ \mp \infty $ for $ a(\cdot) $ follows from the monotonicity of $ a(\cdot) $. 
Now, in view of Proposition \ref{propa}, for all $ t\in\R $,  
\begin{eqnarray}
0 \le a'(t)=\frac{1}{2}(u^2-u_x^2)(t,x(t)+x_+(t)-) & =  & \frac{a(t)}{2}(u-u_x)(t,x(t)+x_+(t)-)\nonumber \\
 &= & \frac{a(t)}{2} (2u(t,x(t)+x_+(t))-a(t)) \; .\label{a'}
\end{eqnarray}
Therefore, since $ a $ takes values in $[\alpha_0/8, 2 \|u_0\|_{H^1} ] $,  it remains to prove that $ a'(t)\to 0 $ as $ t\to \pm\infty $. 
Since
$$
\int_{\R} a'(\tau) \, d\tau 	<\infty \; ,
$$
 the desired result will follow if  $ a' $ is Lipschitz on $ \R $. But this is not too hard to check. Indeed, first from \eqref{esta} we have for all $t\in\R $, $ |a(t)-a(0)|\le t \|u_0\|_{H^1}^2 $ and thus $ t\mapsto a(t) $ is clearly Lipschitz on $ \R $. Second, since $ x(t)+x_+(t)=q^*(t) $, it holds 
$$
\frac{d}{dt} u(t,x(t)+x_+(t))= u(t,q^*(t)) u_x(t,q^*(t))+u_t(t,q^*(t)) \; .
$$
But, $ \sup_{(t,x)\in \R^2} |u u_x |\le 2\|u_0\|_{H^1}^2 $
 and 
 \begin{eqnarray*}
 \sup_{(t,x)\in\R^2}  |u_t| & \le &  \sup_{(t,x)\in\R^2} |u u_x |+  \sup_{(t,x)\in\R^2} \Bigl| (1-\partial_x^2)^{-1}\partial_x (u^2+\frac{1}{2} u_x^2) \Bigr| \\
 & \lesssim &  \|u_0\|_{H^1}^2+ \sup_{t\in\R} \|u^2+u_x^2\|_{L^2_x} \\
& \lesssim & \|u_0\|_{H^1}^2 \; .
 \end{eqnarray*}
 Therefore $ t\mapsto u(t,x(t)+x_+(t)) $ is also Lipschitz on $ \R $ which achieves the proof thanks to \eqref{a'}.
\end{proof}
\noindent
\subsection{ End of the proof of Theorem \ref{liouville}.} 
In this subsection, we conclude by proving that the jump of $ u_x(0,\cdot) $ at $ x(0)+x_+(0) $ is equal to $ -2u(0,x(0)+x_+(0))$. This saturates  for all $ v\in Y_+$, the relation between the jump of $ v_x $  and the value of $ v$ at a point $ \xi \in\R $ and forces $ u(0,\cdot) $ to be equal to $ u(0,x(0)+x_+(0)) \varphi(\cdot- x(0)+x_+(0))$.

We use the invariance of the (CH) equation under the transformation $ (t,x) \mapsto (-t,-x) $. This invariance ensures that  $ v(t,x)=u(-t,-x) $ is also a solution of the (C-H) equation that belongs to $ C(\R; H^1(\R) $, with $ u-u_{xx}\in C_{w}(\R;\M_+) $  and  shares the property of $Y$-almost localization with $ x(\cdot) $ replaced by $ -x(-\cdot) $ and the same fonction $ \varepsilon \mapsto R_\varepsilon $  (See Definition \ref{defYlocalized}). 
Therefore, by applying Propositions \ref{pro3}, \ref{propa} and Lemma \ref{lemmaq} for $ v$ we infer that there exists a $C^1 $-function $ x_- \, :\, \R \mapsto ]-\infty,r_0] $ and a derivable non decreasing function $\tilde{a} \, :\, \R \to [\alpha_0/8, 2 \|u_0\|_{H^1} ]$ with $ \lim_{t\to\mp \infty}
\tilde{a}(t)=\tilde{a}_{\mp} $ 
such that 
\begin{equation}\label{defatilde}
\tilde{a}(t)=v_x(t,(-x(-t)+x_+(t) )+)-v_x(t,(-x(-t)+x_+(t) )-), \quad \forall t\in\R \, .
\end{equation}
Moreover,
$$
\lim_{t\to  \mp\infty} v(t,-x(-t)+x_+(t))  = \lim_{t\to \mp\infty} \tilde{a}(t)/2=\tilde{a}_{\mp}/2\; .\\
$$
Coming back to $ u $ this ensures that 
\begin{eqnarray}
\lim_{t\to  +\infty} u(t,x(t)-x_-(-t)) & =& \lim_{t\to -\infty} \tilde{a}(t)/2=\tilde{a}_{-}/2\; ,\\
\lim_{t\to  -\infty} u(t,x(t)-x_-(-t)) & = & \lim_{t\to +\infty} \tilde{a}(t)/2=\tilde{a}_{+}/2\; ,
\end{eqnarray}
At this stage let us underline that  since
$$
 x_-(-t)=\sup \{ x\in\R,\, \supp y(-t)\in [x(t)-x(-t),+\infty[\} 
 $$
 and $ u\not \equiv 0 $ we must have $ x(-t)+x(t)\ge 0 $ for all $ t\in \R $. 
We claim that this forces 
\begin{equation}\label{aaaa}
\tilde{a}_-=\tilde{a}_+=a_-=a_+ \; .
\end{equation}
Note first that since $ \tilde{a}_-\le \tilde{a}_+ $ and $ a_-\le a_+ $, it suffices to prove that $ \tilde{a}_- \ge a_+ $ and $\tilde{a}_+ \le a_- $. This follows easily by a contradiction argument. Indeed,  assume for instance that $ \tilde{a}_- <a_+$.Then, there exists $ t_0\in \R $ and   $ \varepsilon>0 $ such that $ u(t,x(t)-x_-(-t))<u(t,x(t)+x_+(t)) -\varepsilon $ for all $ t\ge t_0 $.  Since 
 $x(t)-x_-(-t)=q(t-t_0,x(t_0)-x_-(-t_0)) $ and $ x(t)+x_+(t)=q(t-t_0,x(t_0)+x_+(t_0)) $, it follows from \eqref{defq} that 
 $$
 x_+(t)+x_-(-t))\ge \varepsilon (t-t_0) \tendsto{t\to +\infty} +\infty 
 $$
 which contradicts  that $ (x_+(t),x_-(t))\in ]-\infty,r_0]^2 $.  Exactly the same argument but with $ t\to - \infty $ ensures that $\tilde{a}_+ \le a_- $ and completes the proof of the claim \eqref{aaaa}. 

We deduce from \eqref{aaaa} that $ a(t)=a+ $ for all $t\in \R $ and thus \eqref{a'}, \eqref{qq}  and  \eqref{defa} force
$$
u(t,x(0)+x_+(0)+\frac{a_+}{2} t )=\frac{a_+}{2}, \quad \forall t\in \R 
$$
and 
$$
 u_x\Bigl(t,(x(0)+x_+(0)+\frac{a_+}{2}t)-\Bigr)-
u_x\Bigl(t,(x(0)+x_+(0)+\frac{a_+}{2}t)+\Bigr) = a_+, \quad \forall t\in \R \; .
$$
In particular, in view of \eqref{decompositiony},
$$
u(0,x(0)+x_+(0))=\frac{a_+}{2} \mbox{ and } y(0)=a_+ \delta_{x(0)+x_+(0)}+\mu
$$
for some $ \mu\in {\mathcal M}_+(\R) $. But this forces $ \mu=0 $ 
since 
$$
 (1-\partial_x^2)^{-1} (a_+  \; \delta_{x(0)+x_+(0)})=\frac{a_+}{2} \exp\Bigl(-|\cdot -(x(0)+x_+(0))|\Bigr)
 $$
 and for any $ \mu\in {\mathcal M}_+(\R) $, with $ \mu\neq 0 $, it holds 
 $$
  (1-\partial_x^2)^{-1} \nu = \frac{1}{2} e^{-|x|} \ast \nu >0 \mbox{ on } \R \; .
 $$
We thus conclude that $ y(0)=a_+ \delta_{x(0)+x_+(0)} $ which leads to  
$$
u(t,x)=\frac{a_+}{2} \exp \Bigl(-\Bigl|x-x(0)-x_+(0)-\frac{a_+}{2} t\Bigr|\Bigr)
$$
\section{Asymptotic stability of the peakon}\label{5}
Let $ c>0 $ and $ u_0 \in Y_+ $  such that 
\begin{equation}\label{stab}
 \| u_0-c\varphi \|_{H^1} < \Bigl(\frac{\varepsilon^2}{3c^2}\Bigr)^4 \; , \quad 0<\varepsilon<c,
 \end{equation}
 then, according to \cite{CS1}, 
\begin{equation}\label{stabo}
 \sup_{t\in\R} \|u(t)-c\varphi(\cdot-\xi(t))\|_{H^1} <\frac{\varepsilon^2}{c}\; ,
  \end{equation}
 where $ u \in C(\R;H^1) $ is the solution emanating from $ u_0$ and  $ \xi(t)\in\R $ is any point where the function $ u(t,\cdot) $ attains its maximum. By the implicit function theorem, one can prove  the following lemma (see for instance\footnote{In \cite{EL2}, this lemma is stated with $ \varphi '$ instead of $ \rho_{n_0}\ast \varphi' $ in \eqref{ort}. However, there is a gap in the proof since the non smoothness of $ \varphi $ makes the $ C^1$ regularity of $x(\cdot)  $  difficult to prove with this orthogonality condition.} \cite{EL2}) whose proof is postponed to the appendix.
 \begin{lemma}\label{modulation}
 There exists $0< \varepsilon_0<1$,  $\kappa_0>0$, $n_0\in \N $  and $ K>1 $ such that if a solution $ u \in C(\R;Y) $ to 
 \eqref{CH} satisfies 
 \begin{equation}\label{gff}
\sup_{t\in\R}  \|u(t)-c \varphi(\cdot-z(t)) \|_{H^1} < c \varepsilon_0 \; ,
\end{equation}
for some function $ z \; :\; \R\to \R $, then there exists a unique function $ x \; : \R\to \R $ such that 
\begin{equation}
\sup_{t\in\R} |x(t)-z(t)| < \kappa_0 \;  \label{distxz}
\end{equation}
 and 
\begin{equation}
\int_{\R} u(t) (\rho_{n_0}\ast\varphi')(\cdot-x(t))=0, \quad \forall t\in\R \; , \label{ort}
\end{equation}
where $ \{\rho_n\} $ is defined in \eqref{rho} and where 
$ n_0 $ satisfies : 
\begin{equation}\label{unic}
\forall y\in [-1/2,1/2], \quad \int_{\R} \varphi (\cdot-y)  (\rho_{n_0}\ast\varphi')=0 \Leftrightarrow y=0 \; .
\end{equation}
Moreover, $ x(\cdot)\in C^1(\R) $  with 
\begin{equation}\label{estc}
\sup_{t\in\R} |\dot{x}(t)- c| \le \frac{c}{8}
\end{equation}
and if 
  \begin{equation}\label{gf}
\sup_{t\in\R}   \|u(t)-c\varphi(\cdot-z(t)) \|_{H^1} <\frac{\varepsilon^2}{c}=c \Bigl(\frac{\varepsilon}{c} \Bigr)^2
\end{equation}
 for $ 0<\varepsilon <c \varepsilon_0 $ then 
\begin{equation}\label{fg}
\sup_{t\in\R} \|u(t)-c\varphi(\cdot-x(t))\|_{H^1} \le K \varepsilon \; .
\end{equation}
 \end{lemma}
 At this stage, we fix $ 0<\theta<c $ and we take 
 \begin{equation}\label{defep}
 \varepsilon= \frac{1}{2 K}\min \Bigl(\frac{\theta}{2^8}, c \, \varepsilon_0\Bigr)
 \end{equation}
 For $ u_0\in Y_+ $  satisfying \eqref{stab} with this $ \varepsilon$, \eqref{stabo}  ensures that \eqref{gff} and  thus \eqref{estc} hold. Moreover, \eqref{fg} is satisfies with 
 $$ K\varepsilon \le \min \Bigl(\frac{\theta}{2^9},\frac{c \varepsilon_0}{2}\Bigr) \; .
 $$
  It follows that  $ \dot{x} \ge \frac{3}{4} c $ on $ \R $ and that $u $ satisfies the hypotheses of Lemma \ref{almostdecay} for any $ 0<\alpha<1$ such that 
  \begin{equation}\label{okok}
  (1-\alpha)\ge \frac{\theta}{4c} 
  \end{equation}
  and any  $ 0\le \gamma\le (1-\alpha) c $. In particular, $ u$ satisfies the hypotheses of Lemma \ref{almostdecay} for 
   $ \alpha=1/3$. Note that the hypothesis \eqref{difini} with 
 $$
 \eta_0=\frac{1}{K^8}\min\Bigl( \frac{1}{ 2^{10}},  \frac{\varepsilon_0}{6} \Bigr)^8
 $$
 implies that \eqref{stab} holds with $ \varepsilon $ given by \eqref{defep}.

In the sequel we will make use of the following functionals that measure the quantity $ E(u)+\gamma M(u) $ at the right and at the left of $ u$.
For $ 0\le \gamma\le \frac{2c}{3} $, $ v\in Y$ and $ R>0 $ we set 
 \begin{equation}\label{defJr}
 J_{\gamma,r}^{R}(v)=\dist{v^2+v_x^2+\gamma (v-v_{xx})}{\Psi(\cdot -R)} \; .
 \end{equation}
 and
 \begin{equation}\label{defJl}
 J_{\gamma, l}^R(v)=\dist{v^2+v_x^2+\gamma (v-v_{xx})}{(1-\Psi(\cdot+R))} 
 \end{equation}

 Let $ t_0\in\R $ be fixed.  Fixing $\alpha=\beta=1/3 $ and taking $ z(\cdot)=(1-\alpha) x(\cdot)$, $z(\cdot) $ clearly satisfies \eqref{condz}. Moreover, we have  $ J_{\gamma,r}^{R}(u(t_0, \cdot +x(t_0))=I^{+R}_{t_0}(t_0) $
  where $I^{+R}_{t_0} $ is defined in \eqref{defI}.  Since obviously, 
 $$
 J_{\gamma,r}^R \Bigl(u(t,\cdot+x(t))\Bigr)\ge I^{+R}_{t_0}(t)\; , \quad \forall t\le t_0,
 $$
 we deduce from \eqref{mono} that 
 \begin{equation}\label{monoJr}
 J_{\gamma,r}^R \Bigl(u(t_0,\cdot+x(t_0))\Bigr)\le J_{\gamma,r}^R\Bigl(u(t,\cdot+x(t))\Bigr)+K_0 e^{-R/6} \;  , \quad \forall t\le t_0,
 \end{equation}
 where $ K_0 $ is the constant appearing in \eqref{mono}. 
  Now, let us define 
 \begin{eqnarray*}
 \tilde{I}^{R}_{t_0}(t)  & = & \dist{u^2(t)+u_x^2(t)+c y(t)}{1-\Psi(\cdot-x(t)+R+\alpha(x(t_0)-x(t)))}\\
 & =& E(u(t))+cM(u(t)) -I^{-R}_{t_0}(t) \; ,
 \end{eqnarray*}
 where we take again $ z(\cdot)=(1-\alpha) x(\cdot)$. 
 Since $ M(\cdot) $ and $ E(\cdot) $ are conservation laws, \eqref{mono2} leads to 
   $$
 \tilde{I}^{R}_{t_0}(t)\ge \tilde{I}^{R}_{t_0}(t_0)-C e^{-R/6} , \; \forall t\ge t_0 \; .
 $$
 We thus deduce as above  that $ \forall t \ge t_0 $, 
 \begin{equation}\label{monoJl}
J_{\gamma,l}^R \Bigl(u(t,\cdot+x(t))\Bigr)\ge J_{\gamma,l}^R\Bigl(x(t_0,\cdot+x(t_0))\Bigr)-K_0 e^{-R/6} \; .
 \end{equation}

The following proposition proved in the appendix ensures that, for $\varepsilon$ small enough,  the $ \omega$-limit set for the weak $ H^1$-topology of the orbit of $ u_0 $  is constituted by initial data of  $Y$-almost localized solutions. The crucial tools in the proof are the almost monotonicity properties \eqref{monoJr} and \eqref{monoJl}.
\begin{proposition}\label{propasym}
Let $ u_0 \in Y_+$ satisfying \eqref{stab} with $\varepsilon$ defined as in \eqref{defep}  and let $u \in C(\R;H^1(\R)) $ be  the  solution of \eqref{CH}
 emanating from $ u_0$. For any sequence $ t_n\nearrow +\infty $ there exists a subsequence $ \{t_{n_k}\}\subset \{t_n\} $ and $ \tilde{u}_0\in Y_+ $ such that 
\begin{equation}\label{ppp2}
 u(t_{n_k},\cdot+x(t_{n_k})) \weaktendsto{n_k\to +\infty} \tilde{u}_0 \mbox { in } H^1(\R) 
 \end{equation}
 and 
 \begin{equation}\label{pp2}
 u(t_{n_k},\cdot+x(t_{n_k})) \tendsto{n_k\to +\infty} \tilde{u}_0 \mbox { in } H^1_{loc}(\R) 
 \end{equation}
  where $ x(\cdot) $ is a $ C^1$-function satisfying \eqref{ort}, \eqref{estc} and \eqref{fg}. 
Moreover, the solution of \eqref{CH} emanating from $ \tilde{u}_0 $ is $Y$-almost localized.
\end{proposition}
So, let $ u_0 \in Y_+ $ satisfying  \eqref{stab} with $\varepsilon$ defined as in \eqref{defep} and let  $ t_n\nearrow +\infty $ be a sequence of positive real numbers. According to the above proposition, \eqref{ppp2}-\eqref{pp2} hold for some subsequence $ \{t_{n_k}\}\subset \{t_n\} $ and $ \tilde{u}_0\in Y_+ $ such that 
 the  solution of \eqref{CH} emanating from $ \tilde{u}_0 $ is $Y$-almost localized. Theorem 
\ref{liouville} then forces 
$$
 \tilde{u}_0= c_0\varphi (\cdot-x_0) 
 $$
 for some $x_0\in\R $ and $ c_0 $ such that $ |c-c_0|\le K \varepsilon\le c/2^9 $. Note that \eqref{ppp2} together with \eqref{fg} imply $\|c_0 \varphi(\cdot-x_0)-c\varphi \|_{H^1} \le K \varepsilon $ and thus \eqref{defep}  and \eqref{pp2} ensure that $|x_0|\ll 1/2 $. Since by \eqref{ppp2}, $\tilde{u}_0 $ satisfies the orthogonality condition \eqref{ort},
 \eqref{unic} then forces  $ x_0=0 $. On the other hand, \eqref{pp2} and \eqref{fg} ensure that $\displaystyle c_0=\lim_{n\to +\infty} \max_{\R} u(t_{n_k}) $ and thus 
  $$
  u(t_{n_k},\cdot+x(t_{n_k})) -\lambda(t_{n_k})\varphi  \weaktendsto{k\to +\infty} 0 \mbox{ in } H^1(\R)
  $$
where we set $
 \lambda(t):=\max_{\R} u(t) , \quad \forall t\in\R $. Since this is the only possible limit, it follows that 
 $$
  u(t,\cdot+x(t)) -\lambda(t)\varphi  \weaktendsto{t\to  +\infty} 0 \mbox{ in } H^1(\R)\; .
  $$
  and thus 
  \begin{equation}\label{pp3}
 u(t,\cdot+x(t))-\lambda(t)\varphi  \tendsto{t\to 0} 0 \mbox { in } H^1_{loc}(\R) 
 \end{equation}
  \subsection{Convergence in $ H^1(]-A,+\infty[) $ for any $ A>0 $.}\label{51}
 Let $ \delta>0 $ be fixed. Choosing  $ R>0 $ such that $J_{0,r}^R(u(0),\cdot +x(0)) <\delta $ and $K_0 e^{-R/6} \le \delta $,  where $ K_0$ is  the constant that appears in \eqref{monoJr}. We deduce   from \eqref{monoJr}   that 
  $ J_{0,r}^R\Bigl(u(t,\cdot+x(t))\Bigr)<2 \delta $ for all $ t\ge 0 $. This fact together with  the local strong convergence \eqref{pp2} clearly ensure that
  \begin{equation}\label{cvcv}
 u(t,\cdot+x(t))-\lambda(t) \varphi  \tendsto{t\to +\infty}  0\mbox{   in } H^1(]-A,+\infty[)  \mbox{ for any } A>0 \; .
 \end{equation}
 \subsection{Convergence of the scaling parameter}\label{52}
 We claim that 
 \begin{equation}\label{cvlambda}
  \lambda(t)\tendsto{t\to +\infty} c_0\;  .
  \end{equation}
  Let us fix again  $ \delta>0 $ and take  $ R>0 $ such that 
  $ K_0 e^{-R/6} <\delta $. \eqref{monoJl} with $ \gamma=0 $ together with the conservation of $ E(u) $ ensure that, for any pair $ (t,t')\in\R^2$ with $ t>t' $ it holds 
  $$
  \int_{\R} (u^2+u_x^2)(t,x) \Psi(x-x(t)+R) \, dx \le  
      \int_{\R} (u^2+u_x^2)(t',x) \Psi(x-x(t')+R) \, dx+\delta 
 $$
  On the other hand, by the strong convergence \eqref{cvcv} and the exponential localization of $ \varphi , \varphi' $ and $ \Psi $, there exists $ T>0 $ such that 
   for all $ t\ge T $, 
   $$
    \Bigl| \int_{\R} (u^2+u_x^2)(t,x) \Psi(x-x(t)+R) \, dx- \lambda^2(t)E(\varphi) \Bigr| \le \delta \; .
   $$
  It thus follows that 
  $$
  \lambda^2(t) E(\varphi)\le \lambda^2(t') E(\varphi)+3 \delta , \quad \forall t>t'>T \; .
  $$
  Since $ \delta>0 $ is arbitrary, this forces $\lambda $ to have a limit at $ +\infty $ and completes the proof of the claim.
   \subsection{Convergence of $\dot{x} $} \label{53}
   We set  $W(t,\cdot):=c_0\varphi( \cdot-x(t))$ and $ \eta(t)=u(t)-c_0 \varphi(\cdot-x(t))=u(t)-W(t)$ for all $ t\ge 0 $. Differentiating \eqref{ort} with respect to time and using that $ \varphi-\varphi''=2 \delta_0 $, we get 
 $$
\int_{\R} \eta_t  \partial_x W =\dot{ x} \, \langle
\partial_x^2 W \, ,\,   \eta  \rangle_{H^{-1}, H^1}= - 2c_0 \dot{ x} \,  \eta(x(t))+ \dot{ x}\int_{\R}   \eta W, \; 
$$
 and thus
\begin{equation}
\Bigl|\int_{\R}  \eta_t  \partial_x W \Bigr|\le  3 c_0|\dot{ x}-c_0| \| \eta\|_{H^1}+ 2 c_0^2 | \eta(x(t))|+c_0 | \int_{\R}   \eta W|
 \; . \label{huhu}
\end{equation}
Substituting $ u $ by $ \eta+W$ in \eqref{CH}  and
using the  equation satisfied by $W$, we obtain the following equation satisfied by $ v$ :
  \arraycolsep1pt
 \begin{eqnarray}
  \eta_t&  - & (\dot{x}-c_0)  \partial_x W
= -\ \partial_x  \eta W-(1-\partial_x^2)^{-1}\partial_x \Bigl(2 \eta W +  \eta_x W_x\Bigr)\; . \nonumber
 \end{eqnarray}
 \arraycolsep5pt
 At this stage it is worth noticing that \eqref{cvcv}-\eqref{cvlambda} ensures that 
 \begin{equation}\label{fin}
 |\eta(x(t))|+\|\eta_x(t) W(t)\|_{L^2} +\|\eta(t) W(t)\|_{L^2} + \|\eta_x(t) W_x(t)\|_{L^2} \tendsto{t\to +\infty} 0 \; .
 \end{equation}
Taking the $ L^2 $-scalar product with $ \partial_x W$, integrating by parts, using that $ \|\partial_x W\|_{L^2}^2=c_0^2$  and the
 decay of $ \varphi $ and its first derivative,  \eqref{huhu}, \eqref{fin}, \eqref{fg}  and the definition of $ \varepsilon$, we get 
 $$
 |\dot{x}(t)-c_0|\Bigl(c_0^2 -3 c_0 \frac{c}{2^8}  \Bigr) \tendsto{t\to \infty} 0 \; .
 $$
This yields the desired result since $|c-c_0|\le K\varepsilon =\frac{c}{2^8} $ clearly forces $ c\le 2 c_0$. 

   \subsection{Strong $ H^1$-convergence on $ ]\theta t ,+\infty[$}\label{54}
   We deduce  from \eqref{cvlambda} that  
 $$
 u(t,\cdot)-c_0 \varphi (\cdot-x(t)) \weaktendsto{t\to +\infty}  0\mbox{   in } H^1(\R)  
 $$
 and 
 \begin{equation}\label{jw}
 u(t,\cdot+x(t))-c_0 \varphi  \tendsto{t\to +\infty}  0\mbox{   in } H^1(]-A,+\infty[)  \mbox{ for any } A>0 \; .
 \end{equation}
 \eqref{cvforte} will follow by combining these  convergence results with the almost non increasing property  \eqref{mono}. 
 Indeed,  let us fix $ \delta>0$ and take $ R\gg 1 $ such that 
 \begin{equation}\label{sww}
 \|\varphi\|_{H^1(]-\infty,-R/2[}^2< \delta \quad\mbox{and}\quad  \|\Psi-1\|_{L^\infty(]R/2,+\infty[)} <\delta 
\end{equation}
where $ \Psi $ is defined in \eqref{psipsi}. According to the above convergence result there exists $ t_0>0 $ such that 
  $ x(t_0)>R $ and  for all $ t\ge t_0 $, 
 $$
 \|(\eta^2+\eta_x^2)(t,\cdot+x(t)) \|_{H^1(]-R/2,+\infty[)} < \delta \, , 
 $$
 where we set $\eta=u(t)-c_0 \varphi(\cdot-x(t)) $. In particular, \eqref{sww} ensures that
 \begin{equation}\label{swww}
\Bigl| E(\varphi)-\int_{\R} \Bigl( u(t,\cdot+x(t))\varphi +u_x(t,\cdot+x(t))\varphi_x \Bigr) \Psi(\cdot+y) \Bigr| \lesssim \delta, \quad \forall y\ge R,\, \forall t\ge t_0 \; ,
 \end{equation}
 We set  $ z(t)=\frac{\theta}{2} t$ and notice that  \eqref{okok} ensures that \eqref{condz} is satisfied with $ 1-\alpha=\frac{\theta}{4c} $ and $ \beta=1/4$. Moreover, 
  as noticing in the beginning of this section (see \eqref{okok}), $ u $ satisfies the hypotheses of Lemma \eqref{almostdecay} for such $ \alpha $. 
 According to \eqref{mono2} with $ \gamma=0$, we thus get for all $t\ge t_0$,
 $$
 \int_{\R} (u^2+u_x^2)(t,\cdot) \Psi(\cdot-x(t_0)-\frac{\theta}{2}(t-t_0)+R) \le  \int_{\R} (u^2+u_x^2)(t_0,\cdot) \Psi(x-x(t_0)+R )+K_0(\alpha)  e^{-R/6}
 $$
  which leads to  
 \begin{align*}
 \int_{\R} (\eta^2+\eta_x^2)(t,\cdot) & \Psi\Bigl(\cdot-x(t_0)-\frac{\theta}{2}(t-t_0)+R\Bigr)=
 \int_{\R} (u^2+u_x^2)(t,\cdot) \Psi\Bigl(\cdot-x(t_0)-\frac{\theta}{2}(t-t_0)+x_0\Bigr)  \\
 &-2 c_0
 \int_{\R} (u(t) \varphi(\cdot-x(t)) +u_x(t) \varphi_x(\cdot-x(t))  \Psi\Bigl(\cdot-x(t_0)-\frac{\theta}{2}(t-t_0)+R\Bigr)\\
 &+ c_0^2 \int_{\R} (\varphi^2+\varphi_x^2)(t,\cdot-x(t))  \Psi\Bigl(\cdot-x(t_0)-\frac{\theta}{2}(t-t_0)+R\Bigr) \\
 &  \le   \int_{\R} (u^2+u_x^2)(t_0,\cdot) \Psi(\cdot-x(t_0)+R )+K_0(\alpha)  e^{-R/6}\\
 &-2 c_0
 \int_{\R} (u(t_0) \varphi(\cdot-x(t_0)) +u_x(t_0) \varphi_x(\cdot-x(t_0))  \Psi(\cdot-x(t_0)+R )+ C\, \delta \\
 &+ c_0^2 \int_{\R} (\varphi^2+\varphi_x^2)(t_0,\cdot-x(t_0)) \Psi(\cdot-x(t_0)+R )  +C e^{-R/6}\\
 &  \lesssim  \int_{\R}(\eta^2+\eta_x^2)(t,\cdot)  \Psi(\cdot-x(t_0)+R)+C( e^{-R/6}+\delta) \\
 & \lesssim \delta + e^{-R/6}
 \end{align*}
 where in the next to the  last step we used that $ \varphi $ decays exponentially fast and \eqref{swww} since 
  $x(t)-x(t_0)-\frac{\theta}{2}(t-t_0)+R \ge R $ for all $ t\ge t_0$.
  Taking $R$ large enough and $ t_1>t_0$ such that  $ \theta t_1\ge x(t_0)+\frac{\theta}{2} (t_1-t_0)-R $, it follows that for 
  $t\ge t_1 $, 
  $$
 \int_{\R} (\eta^2+\eta_x^2)(t,\cdot)  \Psi(\cdot-\alpha t) \lesssim \delta
 $$
 which completes the proof of Theorem \ref{asympstab} with $ c^*=c_0$.
  \section{Asymptotic stability of train of peakons}\label{6}
In \cite{EL2} the orbital stability in $ H^1(\R) $ of  well ordered trains of peakons is established. More precisely, the following theorem is proved
\footnote{Here again, in  the  statement  given in \cite{EL2}, $ \partial_x \varphi_{c_i} $ appears instead of $ \rho_{n_0} \ast \partial_x \varphi_{c_i} $ in the orthogonality condition \eqref{mod2} and thus there is a gap in the proof of the  $ C^1 $-regularity of the functions $ x_i $, $ i=1,..,N$.  The modifications to get the statement below are exactly the same as the ones to get Lemma  \ref{modulation} that is proven in the appendix.}:
\begin{theorem}[\cite{EL2}]\label{mult-peaks}
Let be given $ N $ velocities $c_1,.., c_N $ such that $0<c_1<c_2<..<c_N $.
There exist   $n_0\in \N $ satisfying \eqref{unic},  $ A>0 $, $ L_0>0 $ and $ \varepsilon_0>0 $ such that if  $ u\in C(\R;H^1) $ is
   the global solution of (C-H) emanating from $ u_0\in Y_+ $, with 
 \begin{equation}
 \|u_0-\sum_{j=1}^N \varphi_{c_j}(\cdot-z_j^0) \|_{H^1} \le \varepsilon^2 \label{ini}
 \end{equation}
 for some  $ 0<\varepsilon<\varepsilon_0$ and $ z_j^0-z_{j-1}^0\ge L$,
with $ L>L_0 $, then there exist $ N $ $C^1$-functions $t\mapsto x_1(t), ..,t \mapsto x_N(t) $ uniquely determined such that
\begin{equation}
\sup_{t\in\R+} \|u(t,\cdot)-\sum_{j=1}^N \varphi_{c_j}(\cdot-x_j(t)) \|_{H^1} \le
A\sqrt{\sqrt{\varepsilon}+L^{-{1/8}}}\;  \label{ini2}
\end{equation}
and 
\begin{equation}
\int_{\R} \Bigl( u(t,\cdot) -\sum_{j=1}^N \varphi_{c_j}(\cdot- x_j(t)) \Bigr)
(\rho_{n_0}\ast \partial_x \varphi_{c_i}) (\cdot - x_i(t)) \, dx = 0 \; , \quad i\in\{1,..,N\}. \label{mod2}
\end{equation}
Moreover,   for $ i=1,..,N $ 
\begin{equation}\label{difdif}
|\dot{x}_i-c_i| \le A \sqrt{\sqrt{\varepsilon}+L^{-{1/8}}}, \quad \forall t\in\R_+  \; .
\end{equation}
 \end{theorem}
 Combining this result with the asymptotic stability of a peakon established in the preceding section, we are able to extend the asymptotic result to a train of well ordered peakons by following the strategy developped in \cite{MMT} (see also \cite{EM}).
 \begin{theorem}\label{asympt-mult-peaks}
Let be given $ N $ velocities $c_1,.., c_N $ such that $0<c_1<c_2<..<c_N $ and $ 0<\theta_0<c_1/4 $.
There exist   $ L_0>0 $
 and $ \varepsilon_0>0 $ such that if  $ u\in C(\R;H^1) $ is
   the solution of (C-H) emanating from $ u_0\in Y_+ $, with 
 \begin{equation}
 \|u_0-\sum_{j=1}^N \varphi_{c_j}(\cdot-z_j^0) \|_{H^1} \le \varepsilon_0^2 
\quad  \mbox{ and } \quad  z_j^0-z_{j-1}^0\ge L_0,\label{inini}
 \end{equation}
then there exist $0< c_1^*<..<c_N^*  $ and $ C^1$-functions $t\mapsto x_1(t), ..,t\mapsto x_N(t) $,  with  $ \dot{x}_j(t) \to c_j^* $ as $ t\to +\infty $, such that,

\begin{equation}\label{mul1}
u(\cdot+x_j(t)) \weaktendsto{t\to +\infty} \varphi_{c_j^*}  \mbox{ in } H^1(\R), \; \forall j\in \{1,..,N\} \; .
\end{equation}

 Moreover, 
\begin{equation}\label{mul2}
u-\sum_{j=1}^N \varphi_{c_j^*}(\cdot-x_j(t)) \tendsto{t\to +\infty} 0 \mbox{ in } H^1(]\theta_0 t,+\infty[)\; .
\end{equation}
 \end{theorem}
 Finally, we will make use of the fact that Camassa-Holm equation possesses special solutions called multipeakons given by 
 $$
 u(t,x)=\sum_{i=1}^N p_i(t) e^{-|x-q_i(t)|} 
 $$
 where $ (p_i(t),q_i(t)) $, $ i=1,..,N$,  satisfy a differential Hamiltonian system (cf. \cite{CH1}). In \cite{Beals0} (see also \cite{CH1}), the limits as $ t\to \mp \infty $ of $p_i(t) $ and $ \dot{q}_i(t) $, $ i=1,..,N$, are determined. Combining the orbital stability of well ordered train of peakons, the continuity with respect to initial data in $ H^1(\R) $ and the asymptotics of multipeakons, the $ H^1$-stability of the variety 
 $$
{\mathcal N}:= \Bigl\{ v=\sum_{i=1}^N p_j e^{-|\cdot-q_j|}, \,
(p_1,..,p_N)\in (\R_+)^N , \, q_1<q_2<..<q_N  \Bigr\} \; .
 $$
 is proved in (\cite{EL2}, Corollary 1.1). Gathering this last result  with the asymptotics of the multipeakons and Theorem \ref{asympt-mult-peaks}, the following asymptotic stability result for not well ordered train of peakons can be deduced quite  directly.
\begin{corollary} \label{cor-mult-peaks}
Let be given $ N $  positive real numbers $ p_1^0,.., p_N^0 $,  $ N
$ real numbers $ q_1^0< ..< q_N^0 $ and let $ 0<\lambda_1<\cdot\cdot<\lambda_N $ be the N distinct eigenvalues of the matrix $ (p_j^0 e^{-|q_i^0-q_j^0|/2})_{1\le i,j\le N} $. 
For any $ B> 0 $ there exists a positive function $ \varepsilon $  with $ \varepsilon(y) \to 0 $ as $ y\to 0 $ and $ \alpha_0>0 $ such that if $ u_0\in
H^1(\R) $ satisfies $ m_0:=u_0-u_{0,xx} \in {\mathcal M}_+(\R) $ with
\begin{equation}
\|m_0\|_{\mathcal M}\le B \quad  \mbox{ and }\quad
\|u_0-\sum_{j=1}^N p_j^0 \exp (\cdot-q_j^0) \|_{H^1}\le \alpha
\label{ini3}
\end{equation}
for some $ 0<\alpha<\alpha_0 $, 
 then there exists $0<c_1^*<\cdot\cdot <c_N^*  $  and $ C^1$-functions $ (x_1,..,x_N) $ with 
 $$
 |c_i^*-\lambda_i|\le \varepsilon(\alpha) \quad \text{and} \quad \lim_{t\to +\infty} \dot{x}_i(t) = c_i^* \; ,\quad \forall i\in \{1,..,N\},
 $$
    such that 
\begin{equation}\label{coromul2}
u-\sum_{i=1}^N \varphi_{c_i^*}(\cdot-x_i(t)) \tendsto{t\to +\infty} 0 \mbox{ in } H^1(]\frac{\lambda_1 t }{4} ,+\infty[)\; .
\end{equation}
\end{corollary}
 \noindent 
 {\bf Proof of Theorem \ref{asympt-mult-peaks}.}
 We first concentrate ourself on the fastest bump located around $ x_N(\cdot) $. To adapt Lemma \ref{almostdecay} to this bump, 
 we define $ I^{N,x_0}_{t_0} (\cdot) $ as $ I^{x_0}_{t_0} (\cdot) $ in Lemma \ref{almostdecay} but with $ z_{t_0}^{x_0}(\cdot) $ replaced by 
$$ z^{N,x_0}_{t_0} =x_N(t_0)+x_0+z(t)-z(t_0) \; ,
 $$
 where $ z \, :\, \R\to \R $ is a $ C^1$-function that satisfies 
 \begin{equation}\label{za}
 (1-\alpha) \dot{x}_N(t) \le \dot{z}(t) \le (1-\beta) \dot{x}_N(t) 
 \end{equation}
 for some $ 0	<\beta\le \alpha	<1 $. 
 We start by noticing that to prove \eqref{mono} we can replace \eqref{loc} by the less restrictive condition :
$$
 \|u(t)\|_{L^\infty(x-x(t)>R_0 )} \le \frac{ (1-\alpha) c_0}{2^6} \; .
$$
Indeed, it is direct to check that  the condition $ x-x(t)\le R_0 $  is sufficient to get  \eqref{to1}.  Therefore, to get  the same estimate as 
 \eqref{mono} for $ I^{N,R}_{t_0}(\cdot) $ it suffices to assume that  there exists $ R_0>0 $ and $ 0<\alpha<1 $ such that 
 \begin{equation}\label{locbis1}
 \|u(t)\|_{L^\infty(x-x_N(t)>R_0 )} \le \frac{ (1-\alpha) c_N}{2^6}\, ,  \; \forall t\in\R .
 \end{equation}
Now, because of the presence of the bumps at the left of the Nth bump, we are able to  establish the same estimate as \eqref{mono2} only for $ t_0 \ge t_R $ where $ t_R $ depends on $ R$. To prove such version of \eqref{mono2} with $ \alpha= \frac{5}{8} \frac{c_N-c_{N-1}}{c_N} $ we replace \eqref{loc} by 
\begin{equation}\label{con6}
\|u(t)\|_{L^\infty (]\frac{5x_{N-1}(t)+x_N(t)}{6}, x_N(t)-R_0[)} <\frac{(1-\alpha)c_N}{2^6} , \quad \forall t \ge t_R ,
\end{equation} 
with $ x_N(t)-x_{N-1}(t)\ge 2 R $. For any   $ R\ge R_0 $ we define 
\begin{equation}\label{deftR}
t_R=0\vee \{t\ge 0, \; x_N(t)-x_{N-1}(t)=2R\}   \; .
\end{equation}
For $ R\ge R_0$, assuming that $z(\cdot) $ satisfies \eqref{za} and that 
\begin{equation}\label{fin2}
|c_N-\dot{x}_N(t)|+|c_{N-1}-\dot{x}_{N-1}(t)| \le \frac{1}{12} (c_N-c_{N-1}) \, ,\; \forall t \ge 0 , 
\end{equation}
  we get for $ x\le \frac{5x_{N-1}(t)+x_N(t)}{6} $ and $ t_0\ge t_R $, 
\begin{align*}
x-z^{N,-R}_{t_0}&=x-x_N(t)+R +(x(t)-z(t))-(x(t_0)-z(t_0))\\
& \le -\frac{5}{6} (x_N(t)-x_{N-1}(t))+R +\alpha c_N (t-t_0) \\
& \le  -\frac{5}{3} R -\frac{3}{4} (c_N-c_{N-1})(t-t_0) +R +\frac{5}{8} (c_N-c_{N-1})(t-t_0)\\
& \le -\frac{2}{3} R -\frac{1}{8} (c_N-c_{N-1})(t-t_0)\, , 
\end{align*}
where we took $  \alpha= \frac{5}{8} \frac{c_N-c_{N-1}}{c_N}$. 
This leads to 
$$
\Psi(x-z^{N,-R}_{t_0}) \lesssim e^{-\frac{R}{9}} e^{-\frac{1}{48}(c_N-c_{N-1})(t-t_0)}\; .
$$
which is sufficient to get \eqref{J11} with $ \beta c_0 $ replaced by $\frac{c_N-c_{N-1}}{48}$.

In the sequel, we  set 
\begin{equation}\label{defL}
\varepsilon_0=\Bigl( \frac{\sigma_0}{2^{18}}\Bigr)^8  \mbox{ and } L_0=\Bigl( \frac{\sigma_0}{2^{18}}\Bigr)^{16}
\quad \text{ where } \sigma_0=A\,  \Bigl( \min_{i=2,..,N} (c_i-c_{i-1})\wedge \theta\Bigr) \, .
\end{equation}
Taking $ \alpha=\frac{5}{8} \frac{c_N-c_{N-1}}{c_N} $, we infer from \eqref{ini2} that for $ R_0 $ such that 
\begin{equation}\label{fin3}
N C_N e^{-R_0} <\frac{\sigma_0}{2^{18}}
\end{equation}
\eqref{locbis1} is satisfied . Moreover, for $R  \ge R_0 $, \eqref{difdif} ensures \eqref{fin2} is satisfied and  that \eqref{con6} is satisfied with $ t_R$ defined as in \eqref{deftR}.
 Therefore, taking  $ z(\cdot)=(1-\alpha) x_N(\cdot) $  and $ \beta=\alpha= \frac{5(c_N-c_{N-1})}{ 8 c_N}$, we infer that for any $ R>R_0 $ 
  \begin{equation}\label{mono3}
 I_{t_0}^{N,+R}(t_0) - I_{t_0}^{N,+R}(t) \le K_0 e^{-R/6} ,\quad  \forall 0\le  t\le t_0
\end{equation}
 and 
  \begin{equation}\label{mono4}
 I_{t_0}^{N,-R}(t) -  I_{t_0}^{N,-R}(t_0) \le K_0 e^{-\frac{R}{24}} , \quad  \forall t\ge t_0\ge t_R
\end{equation}
where $ t_R $ is defined as in \eqref{deftR} and where $ K_0=K_0(\sigma_0) $. 
  As in Section \ref{5}, this ensures that
 $$
  J_{\gamma,r}^R \Bigl(u(t_0,\cdot+x_N(t_0))\Bigr)\le J_{\gamma,r}^R\Bigl(u(t,\cdot+x_N(t))\Bigr)+K_0 e^{-R/6} \;  , \quad \forall 0\le t\le t_0,
 $$
 and 
 $$
   J_{\gamma,l}^R \Bigl(u(t_0,\cdot+x_N(t_0))\Bigr)\ge J_{\gamma,l}^R\Bigl(u(t,\cdot+x_N(t))\Bigr)-K_0 e^{-\frac{R}{24}}\;  , \quad \forall t \ge t_0\ge t_R\; .
 $$
 Since, we only need these last  two estimates in the proof of Proposition \ref{propasym} as well as in Subsections  \ref{51}-\ref{53}, we infer  that there exists $ c_N^* $ close to $ c_N $ such that $ \dot{x}_N \to c_N^*$ as $ t\to +\infty $ and 
 $$
  u(t,\cdot+x_N(t))\weaktendsto{t\to +\infty} c_N^* \varphi \mbox{   in } H^1(\R) \; .
 $$
 Moreover,
  \begin{equation}\label{cvcvN}
 u(t,\cdot+x_N(t))-c_N^* \varphi  \tendsto{t\to +\infty}  0\mbox{   in } H^1(]-A,+\infty[)  \mbox{ for any } A>0 \; .
 \end{equation}
Now, setting
   $y_N=\frac{x_N+x_{N-1}}{2} $ and noticing that for all $ t\ge 0 $, $ z(\cdot)=y_N(\cdot)   $ also satisfies \eqref{condz} with  
   $ \alpha=  \frac{5(c_N-c_{N-1})}{ 8 c_N}$ and $ \beta = \frac{(c_N-c_{N-1})}{ 8 c_N}$, we get that for $ R\ge R_0 $, 
  \begin{equation}\label{yN}
    \int_{\R} (u^2+u_x^2)(t,\cdot) \Psi(\cdot-y_N(t))=I^{N,y_N(t_R)-x_N(t_R)}_{t_R} (t)
  \end{equation}
   is also almost non increasing for $ t \ge t_{R}$ where $ t_{R} $ is defined in \eqref{deftR}. Indeed, we have $ x_N(t_R)-y_{N}(t_R)\ge R \ge R_0$.
   
   This enables, as in Subsection \ref{53}, to prove that actually
   $$
    \int_{\R} \Big[\Bigl(u(t)- c_N^* \varphi(\cdot-x_N(t))\Bigr)^2+\Bigl(u_x- c_N^*\varphi_x (\cdot-x_N(t))\Bigr)^2 \Bigr]\Psi(\cdot-y_N(t)) \tendsto{t\to \infty}  0\; .    $$
    Let us  now set  $ y_i=(x_{i}+x_{i-1})/2 $ for $ i=2,..,N-1 $ and $y_1(t)=\theta t $. We claim that if  for some $ 1\le i\le N-1 $ it holds
  \begin{equation}\label{hyp}
  \int_{\R} \Big[\Bigl(u(t)-\sum_{j=i+1}^N c_j^* \varphi(\cdot-x_N(t))\Bigr)^2+\Bigl(u_x-\sum_{j=i+1}^N c_j^* \varphi_x (\cdot-x_N(t))\Bigr)^2 \Bigr]\Psi(\cdot-y_{i+1}(t)) \tendsto{t\to \infty}  0 
  \end{equation}
  then  $ \dot{x}_i(t)\to c_i^* $ as $ t\to \infty $ for some $ c_i^* $ close to $c_i $ and 
  \begin{equation}\label{hyp2}
   \int_{\R} \Big[\Bigl(u(t)-\sum_{j=i}^N c_j^* \varphi(\cdot-x_N(t))\Bigr)^2+\Bigl(u_x-\sum_{j=i}^N c_j^* \varphi_x (\cdot-x_N(t))\Bigr)^2 \Bigr]\Psi(\cdot-y_i(t)) \tendsto{t\to \infty}  0
  \end{equation}
  which clearly yields the desired result by a finite iterative argument.

 We start by  noticing that for $ i\in \{2,..,N-1\}$, \eqref{con6} also holds for $ u $ with $ \alpha= \frac{5(c_i-c_{i-1})}{ 8 c_i} $ and with 
   $x_N, x_{N-1} $, $ c_N$, $ c_{N-1}$ and $ t_R$ replaced by  respectively $ x_i $, $ x_{i-1} $,  $ c_i $ and $ c_{i-1}$ and 
    \begin{equation}\label{deftRi}
t_R^i =\max\Bigl(\{0\}\cup \{t\ge 0, \; x_i(t)-x_{i-1}(t)=2R\}  \Bigr) \; .
\end{equation}
  Moreover, for $ i=1 $,  \eqref{defL} and \eqref{ini2} ensure that 
  \begin{equation}\label{i=1}
  \|u(t)\|_{L^\infty(]-\infty, x_1(t)-R_0[} < \frac{1-\alpha}{2^6} c_1
  \end{equation}
  with $ 1-\alpha=\frac{\theta}{4c_1} $.
  Therefore, defining $ I^{i,x_0}_{t_0} $ in the same way as $ I^{x_0}_{t_0} $ but with $ x(\cdot) $ replaced by $x_i(\cdot) $ and taking  $ z(\cdot)=(1-\alpha) x_i(\cdot) $  with $ \alpha=\beta= \frac{5(c_i-c_{i-1})}{ 8 c_i}$ if $ i\ge 2$ and  $z(t)=\frac{\theta}{2} t $, $
   \alpha=1-\frac{\theta}{4c_1} $, $\beta =1/4$ if $ i=1$, we get that for any $ R\ge R_0 $ it holds
 \begin{equation}\label{mono5}
 I_{t_0}^{i,-R}(t) - I_{t_0}^{i,-R}(t_0) \le K_0  e^{-\frac{R}{24}}  ,\quad \forall i\in\{1,..,N-1\},\; \forall  t\ge t_0\ge t_R^i \; ,
\end{equation}
 where $ K_0=K_0(\sigma_0) $.
As in Section \ref{5}, it follows that for $ \gamma\ge 0 $ small enough,
 \begin{equation}\label{monoJl1}
J_{\gamma,l}^{R} \Bigl(u(t,\cdot+x_i(t))\Bigr)\ge J_{\gamma,l}^{R}\Bigl(u(t_0,\cdot+x_i(t_0))\Bigr)-K_0 e^{-\frac{R}{24}}   , \; \forall t\ge t_0\ge t_R^{i}\; .
 \end{equation}
 Now, the proof of the almost monotonicity of  $J_{\gamma,r}^R \Bigl(u(t,\cdot+x_i(t))\Bigr) $ is more subtle. Indeed, starting at $ x_i(t_0)+R $ 
  at time $ t_0 $ for some $ R>0 $ and traveling back in time with a fixed speed strictly less than $ c_i $, one will cross $ x_{i+1}(\cdot) $ at some time $ t$ which will tend to $ +\infty $ if $ t_0 $ tends to $ +\infty $. This is clearly not allowed if we want to prove an almost monotonicity result. 
  To  overcome this difficulty we will decompose the travel back into  two parts. First, one travels back with some speed strictly less than $ c_i $ till one crosses the   curve of the middle point   $ y_{i+1}(\cdot) =(x_i(\cdot ) +x_{i+1}(\cdot))/2 $. Then, one continues to travel back but along $ y_{i+1}(\cdot) $ until  the time $ t^{i+1}_R $ that satisfies
    $ x_{i+1}(t^{i+1}_R)-x_{i}(t^{i+1}_R)\ge 2 R $. This is the idea of the proof of the 
    following lemma which ensures that  $J_{\gamma,r}^R \Bigl(u(t,\cdot+x_i(t))\Bigr) $ is  almost non increasing for $ t\ge t_R^{i+1} $. 
\begin{lemma} \label{yvan}
Let $ i\in\{1,.., N-1\}$ and $ 0\le \gamma \le c_i $. For any $ R>0 $ it holds
\begin{equation}\label{monoJr1}
J_{\gamma,r}^R \Bigl(u(t_0,\cdot+x_i(t_0))\Bigr)\le J_{\gamma,r}^R\Bigl(u(t,\cdot+x_i(t))\Bigr)+K_0 e^{-\frac{R}{24}}   , \; \forall t_0\ge t\ge t_R^{i+1}\; .
 \end{equation}  
 where $ t^i_R $ is defined as in \eqref{deftRi}.
\end{lemma}
\begin{proof} Let $ R>0 $ and $ t_0>t^i_R $. 
We set 
$$
t_0'=t_R^{i+1} \vee  \{t\in [t_R^{i+1},t_0] , \, x_i(t_0)+R +\frac{3}{4} (x_i(t)-x_i(t_0))=y_{i+1}(t)\} 
$$
On $ [t_0',t_0] $ it holds 
$$
x_i(t_0)+R +\frac{3}{4} (x_i(t)-x_i(t_0)) \le y_{i+1}(t) 
$$
and thus seting $ z^{i,R}_{t_0}(t)=x_i(t_0)+R +\frac{3}{4}(x_i(t)-x_i(t_0)) $,  we get for any $ t\ge t_R^{i+1} $ and 
 any $ x\ge \frac{5x_{i+1}(t)+3x_{i}(t)}{8}  $ that 
 $$
 x-z^{i,R}_{t_0}(t) \ge x-y_{i+1}(t) \ge \frac{R}{4} +\frac{1}{2^4} (c_{i+1}-c_i) (t-t_R^{i+1}) 
 $$
 which leads to 
 $$
 \Psi(x-z^{i,R}_{t_0}(t) ) \le e^{-\frac{R}{24}} e^{-\frac{1}{2^7} ( c_{i+1}-c_i) (t-t_R^i) }\; .
 $$
 On the other, \eqref{defL} and \eqref{ini2} ensure that for $ R\ge R_0 $ (with $ R_0 $ defined in \eqref{fin3}), 
\begin{equation}\label{con66}
\|u(t)\|_{L^\infty (]x_i(t)+R,  \frac{5x_{i+1}(t)+3x_{i}(t)}{8}[} <\frac{\theta}{2^8}< \frac{1}{2} \frac{c_i}{2^6} , \quad \forall t \ge t_R^{i+1} ,
\end{equation}  
 Therefore, defining 
 $$
 I_{t_0}^{i,R}(t)=\dist{u^2(t)+u_x^2(t)+\gamma  y(t)}{ \Psi(\cdot -z^{i,R}_{t_0}(t))},
  $$
  with $ 0\le \gamma \le c_i $, 
 we get as in \eqref{mono} that for all $ t \in [t_0',t_0] $,
 $$
 I_{t_0}^{i,R}(t_0)-I^{i,R}_{t_0}(t) \le K_0 e^{-\frac{R}{24}} \; .
 $$
 It follows as in \eqref{monoJr} that 
 $$
 J^R_{\gamma,r}(u(t,\cdot+x_i(t_0))\le J^R_{\gamma,r}(u(t,\cdot+x_i(t))+K_0 e^{-\frac{R}{24}}  , \quad \forall t\in [t_0',t_0] 
 $$
 If $ t_0'=t_R^{i+1} $ we are done. Otherwise we must have $ z^{i,R}_{t_0'} =y_{i+1}(t_0') $. But then the same arguments as in \eqref{yN} lead, for $ 0\le \gamma \le c_i $,  to 
 $$
\dist{u^2(t)+u_x^2(t)+\gamma  y(t)}{ \Psi(\cdot -y_{i+1}(t_0'))} \le
\dist{u^2(t)+u_x^2(t)+\gamma  y(t)}{ \Psi(\cdot -y_{i+1}(t))} +K_0 e^{-R/9} , \; \forall t\in [t_R^{i+1}, t_0']\; .
 $$
 Since for $ t\ge t_R^{i+1} $,
 $$
\dist{u^2(t)+u_x^2(t)+\gamma  y(t)}{ \Psi(\cdot -y_{i+1}(t))}\le J_{\gamma,r}^R (u(t,\cdot+x_i(t)) 
  $$
  we obtain the desired result for all $ t\in [t_R^{i+1},t_0] $. 
 \end{proof}

 Since the proof of Proposition \ref{propasym} only uses the almost monotonicity of $J_{\gamma,r}^R(u(t,\cdot+x(t))) $ and  $J_{\gamma,l}^R(u(t,\cdot+x(t))) $ for 
   some 
  $ \gamma>0 $ and for $ t\ge t_R $ where $ t_R $ is a non negative time depending on $ R $, we obtain as in \eqref{pp3} that 
  \begin{equation}\label{ddd}
  u(t,\cdot+x_i(t))-\lambda_i(t) \varphi \tendsto{t\to +\infty} 0 \mbox{   in } H^1_{loc}(\R) \; .
 \end{equation}
 where
 $$
  \lambda_i(t)=\max_{x\in [-c_i, c_i]} u(t,x+x_i(t)) \, .
  $$
  Let us now set 
$$
 W_{>i}(t)=\sum_{j=i+1}^N c_j^* \varphi(\cdot -x_j(t))\mbox{ and }  v=u-W_{>i}(t) \, ,
 $$
 In view of \eqref{difdif}, for any  $\varepsilon>0 $, there exists $ t_\varepsilon>0 $  such that for all  $ R\ge 0 $ and $ t\ge t_\varepsilon $, 
 $$
 \Bigl|   J_{0,l}^R \Bigl(u(t, \cdot +x_i(t))-   J_{0,l}^R \Bigl(v(t, \cdot +x_i(t)) \Bigr| \le \varepsilon\; .
 $$
 Moreover, decomposing $  J_{0,r}^R \Bigl(u(t,\cdot+x_i(t))\Bigr) $ as 
 \begin{align*}
  J_{0,r}^R \Bigl(u(t,\cdot+x_i(t))\Bigr)& =\int_{\R} (u^2+u_x^2)\Psi(\cdot -x_i(t)-R) \Bigl[ 1-\Psi\Bigl(\cdot-y_{i+1}(t)\Bigr)\Bigr]\\
 &  +\int_{\R}(u^2+u_x^2) \Psi(\cdot -x_{i+1}(t)-R)\Psi\Bigl(\cdot-y_{i+1}(t)\Bigr)\\
 & =A_1(t)+A_2(t)\; ,
  \end{align*}
   \eqref{difdif}  ensures that 
   $$
   A_1(t)-\int_{\R} (v^2+v_x^2)\Psi(\cdot -x_i(t)-R) \Bigl[ 1-\Psi\Bigl(\cdot-y_{i+1}(t)\Bigr)\Bigr]\tendsto{t\to \infty} 0 
   $$
   and \eqref{hyp} together with  \eqref{difdif} ensures that 
     $$
   A_2(t)- \int_{\R}(v^2+v_x^2) \Psi(\cdot -x_{i}(t)-R)\Psi\Bigl(\cdot-y_{i+1}(t)\Bigr)\tendsto{t\to \infty} \sum_{i+1}^N E(\varphi_{c_i^*}) \; . 
   $$
   This proves that $J_{0,r}^R(v(t,\cdot+x_i(t))) $ and  $J_{0,l}^R(v(t,\cdot+x_i(t))) $ enjoy the same almost monotonicity property as respectively 
   $J_{0,r}^R(v(t,\cdot+x_i(t))) $ and  $J_{0,l}^R(v(t,\cdot+x_i(t))) $ for $t\ge t_R $ large enough. 
   
     Let now $ \delta>0 $ be fixed. According to   \eqref{hyp} , there exists $ t_\delta>0 $ such that 
  $$
  \int_{\R} (v^2+v_x^2)(t_\delta,x)\Psi(x-y_{i+1}(t_\delta))\, dx <\delta/3 \; .
  $$
  Moreover, by \eqref{difdif}, we may  also require
   $$
  \int_{\R} (\varphi^2+\varphi_x^2) (x-x_i(t)) \Psi(x-y_{i+1}(t_\delta))\, dx <\delta/3 \; .
  $$
Therefore using the  almost monotonicity of  $J_{0,r}^R(v(t,\cdot+x_i(t))) $, with $ R=y_{i+1}(t_\delta)-x_i(t_\delta) $, together with the  local strong convergence result  \eqref{ddd} and \eqref{difdif}, we get that for all fixed $ A>0 $, 
\begin{equation}\label{oi}
u(t,\cdot+x_i(t))-\lambda_i(t) \varphi- W_{>i}(t,\cdot +x_i(t))\tendsto{t\to +\infty} 0 \mbox{ in } H^1(]-A,+\infty[) \; .
\end{equation}
To prove the convergence of the scaling parameter $ \lambda_i$ we use   the above  strong convergence in  $ H^1(]-A,+\infty[) $ and \eqref{difdif}  to get that for any $\delta>0 $ there exists $R_{\delta}>0 $ and $t_{\delta}>0 $ such that  
  $$
    \Bigl| \int_{\R} (v^2+v_x^2)(t,x) \Psi(x-x_i(t)+R_\delta) \, dx- \lambda_i^2(t)E(\varphi) \Bigr| \le \delta \, , \quad \forall t\ge t_{\delta} \, .
   $$
   The almost monotonicity of $J_{0,l}^{R_\delta}(v(t,\cdot+x_i(t))) $ then ensures that $ \lambda_i(t) \to c_i^* $ as $ t\to +\infty$, for some $ c_i^* $ close to $ c_i $. Hence, we get 
     \begin{equation}\label{oi2}
u(t,\cdot+x_i(t))-c_i^* \varphi- W_{>i}(t,\cdot +x_i(t))\tendsto{t\to +\infty} 0 \mbox{ in } H^1(]-A,+\infty[) \; .
\end{equation}
    To get the convergence of $ \dot{x}_i $ towards $ c_i^* $ we write the equation for 
    $$
     \eta=u-\sum_{j=i}^N c_j^* \varphi(\cdot-x_j(t)) =u-W_{\ge i}=v-c_i^* \varphi(\cdot -x_i(t)) 
     $$
      and proceed as in Subsection \ref{53}.
    
    Finally to prove \eqref{hyp2}  we first notice that,  proceeding as in Subsection \ref{54}, we obtain that for any fixed $ \delta>0 $, there exists $ R>0 $ and $t_0\ge 0 $ such that
   $$
   \Bigl| \sum_{j\ge i}^N E(\varphi_{c_{j^*}})-\int_{\R} \Bigl( W_{\ge i}  +(\partial_x  W_{\ge i})^2 \Bigr) \Psi(\cdot-x_i(t)+y) \Bigr| < \delta, \quad \forall y\ge R,\, \forall t\ge t_0 \, ,
 $$
 $$
 \|(\eta^2+\eta_x^2)(t,\cdot+x_i(t)) \|_{H^1(]-2R,+\infty[)} < \delta \, ,\quad  \forall t\ge t_0 \; ,
 $$
 and 
 $$
\Bigl|\sum_{j\ge i}^N E(\varphi_{c_{j^*}})-\int_{\R} \Bigl( u(t,\cdot+x_i(t)) W_{\ge i} +u_x(,\cdot+x_i(t))\partial_x W_{\ge i}  \Bigr) \Psi(\cdot+y) \Bigr| < \delta, \quad \forall y\ge R,\, \forall t\ge t_0 \; .
 $$
 Therefore, using the almost monotonicity of $ t\mapsto I^{i,-R}_{t_0} (t)$ with $ \gamma=0 $,  $z(t)=\frac{1}{2}(x_{i-1}(t)+y_i(t)) $,   $
 (1-\alpha) =\frac{c_{i}+7c_{i-1}}{8c_i} $ if $ i\ge 2$ and $z(t)=\frac{\theta}{2}t $, $\alpha=1-\frac{\theta}{4c_1} $ and  $ \beta =1/4 $ if $ i=1$, and proceeding as in Subsection \eqref{54}, we get that for $ t\ge t_0 $,
     \begin{align*}
 \int_{\R} (\eta^2+\eta_x^2)(t,\cdot) & \Psi(\cdot-z^{i,-R}_{t_0}(t))=
 \int_{\R} (u^2+u_x^2)(t,\cdot) \Psi(\cdot-z^{i,-R}_{t_0}(t))  \\
 &-2 c_i^*
 \int_{\R} \Bigl(u(t) W_{\ge i} (t) +u_x(t) \partial_x W_{\ge i} (t)\Bigr) \Psi(\cdot-z^{i,-R}_{t_0}(t))\\
 &+ (c_i^*)^2 \int_{\R} (W_{\ge i}^2+(\partial_x W_{\ge i})^2)(t)  \Psi (\cdot-z^{i,-R}_{t_0}(t)) \\
 & \le  \int_{\R} (v^2+v_x^2)(t_0,\cdot) \Psi(\cdot-x_i(t_0)+R)+K_0(\sigma_0)  e^{-R/6}\\
 &-2 c_i^*
 \int_{\R} \Bigl(u(t_0) W_{\ge i}(t_0) +u_x(t_0) \partial_x W_{\ge i}(t_0) \Bigr) \Psi(\cdot-x_i(t_0)+R )\\
 &+ (c_i^*)^2 \int_{\R} (W_{\ge i}^2+(\partial_x W_{\ge i})^2)(t_0)  \Psi(\cdot-x_i(t_0)+R) +2\delta\\
 & \le  \int_{\R}(\eta^2+\eta_x^2)(t,\cdot)  \Psi(\cdot
 -x_i(t_0)+R)+K_0(\sigma_0)e^{-R/6}+2\delta\\
& \lesssim \delta +e^{-R/6}\, ,
 \end{align*} 
 where  $ z^{i,-R}_{t_0}(t)=x_i(t_0)-R+z(t)-z(t_0)\le x_i(t)-R $, $ \forall t\ge t_0$. This yields the result since, $ R $ being fixed, it holds  $z^{i,-R}_{t_0}(t)\le y_i(t) $ for $t $ large enough.
    \section{Appendix}
    \subsection{Proof of \eqref{go}}\label{sect71}
   Using that 
  \begin{equation}
    \frac{d}{dt} \int_{\R} u^2+u_x^2 = 2 \int_{\R} u u_t g +2 \int_{\R} u_x u_{xt} g \; , \label{z0}
    \end{equation}
  \eqref{CH} yields
   \begin{align}
    2 \int_{\R} u u_t g & =  -2 \int_{\R} u^2 u_x g -2 \int_{\R} u \, p_x \ast (u^2+u_x^2/2) g \nonumber \\
     & =  \frac{2}{3} \int_{\R} u^3 g' -2 \int_{\R} u \, p_x \ast (u^2+u_x^2/2) g \label{z1}
    \end{align}
    and, recalling  that $p_{xx}=p-\delta_0 $, 
      \begin{align}
    2 \int_{\R} u_x u_{xt} g & =  -2 \int_{\R} u_x^3 g -2\int_{\R} u_x u u_{2x} g -2 \int_{\R} u_x  p_{xx}  \ast (u^2+u_x^2/2) g \nonumber \\
     & = -2 \int_{\R} u_x ^3 g +\int_{\R} u_x ^3 g+\int_{\R}  u u_x^2 g'\nonumber \\
     & \quad -2 \int_{\R} u_x p \ast (u^2+u_x^2/2) g+2\int_{\R} u_x (u^2+u_x^2/2) g \nonumber \\
     & =  \int_{\R}  u u_{x}^2 g'+2 \int_{\R} u\,  p \ast (u^2+u_x^2/2) g'+2 \int_{\R} u \, p_x  \ast (u^2+u_x^2/2) g-\frac{2}{3} \int_{\R} u^3 g'\; .
     \label{z2}
    \end{align}
    Gathering \eqref{z0}-\eqref{z2}, \eqref{go} follows.
     \subsection{Proof of  Lemma \ref{modulation}}

 Let $ n_0\in \N $  to be specified later.  For  $z\in\R $ we define the function 
   $$Y_z \; :\; \begin{array}{rcl}
  \R \times H^1(\R) & \longrightarrow & \R \\
   (y,v) & \mapsto & \displaystyle \int_{\R} v(x) (\rho_{n_0}\ast \varphi)'(x-z-y) \, dx 
   \end{array} \; .
   $$
 Since $ \rho_{n_0} $ and $ \varphi $ are both even, one has $ Y_z(0,\varphi(\cdot -z))=0 $.  Moreover,  $ Y $ is clearly of class $ C^1 $   and
 it holds 
 \begin{equation}\label{vb0}
  \frac{\partial Y_z}{\partial y} (0,\varphi(\cdot-z)) =  \int_{\R} \varphi' (\rho_{n_0} \ast \varphi')=\|\varphi'\|_{L^2}^2-\varepsilon(n_0)=1 -\varepsilon(n_0)\; ,
  \end{equation}
  with $ \varepsilon(n)\to 0 $ as $ n \to +\infty $. Therefore by taking $ n_0 $ large enough, we may require  that 
  $$
  \frac{\partial Y_z}{\partial y} (0,\varphi(\cdot-z)) \ge 1/2 \; .
  $$
  From the implicit function theorem we deduce that there exists $ \tilde{\varepsilon}_0>0 $, $\kappa_0>0 $  and a $ C^1 $-function $ y_z $ from 
  $ B_{H^1}(\varphi(\cdot-z),\tilde{\varepsilon}_0) $
  in $]\kappa_0,\kappa_0 [$  which is uniquely determined such that
   $$
   Y_z(y_z(u),u)=Y(0,\varphi)=0 , \quad \forall u\in B_{H^1}(\varphi(\cdot-z),\tilde{\varepsilon}_0)\; .
   $$
 In particular there exists $ C_0>0 $ such that if $ u\in B_{H^1}(\varphi(\cdot-z),\beta) $ with $ 0<\beta\le \tilde{\varepsilon}_0 $ then 
 \begin{equation}\label{rd}
 |y_z(u)|\le C_0 \beta\; .
 \end{equation}
 Note that, by a  translation symmetry argument, $\tilde{\varepsilon}_0 $, $\kappa_0 $ and $ C_0 $ are independent of $ z\in \R $. Therefore, by uniqueness, we can define a $ C^1 $-mapping $ \tilde{x} $ from 
 $ U_{z\in\R}  B_{H^1}(\varphi(\cdot-z),\tilde{\varepsilon}_0) $ into $]\kappa_0,\kappa_0 [$ by setting 
  $$
  \tilde{x}(u)=z+y_z(u) \; \text{ for   } u\in    B_{H^1}(\varphi(\cdot-z),\tilde{\varepsilon}_0)  \; .
  $$
 Now we notice that $ Y_z $ is also  a $ C^1$-function from $\R\times L^2(\R) $ into $ \R $ with
 $$
  \frac{\partial Y_z}{\partial y} (y,\varphi(\cdot-z)) =  \int_{\R} u(x)(\rho_{n_0}'' \ast \varphi)(x-z-y)\, dx \; . $$
  Therefore, in the same way as above we obtain that there exists $\tilde{\tilde{\varepsilon}}_0>0 $ and a   $ C^1 $-function $ \tilde{\tilde{x}} $ from  $ \cup_{z\in\R} B_{L^2}(\varphi(\cdot-z),\tilde{\tilde{\varepsilon}}_0) $ 
   into a neighborhood of $ 0 $ in $ \R $ such that 
  $$
  \int_{\R} u (\rho_{n_0}\ast \varphi)'(\cdot-y) =0 \Leftrightarrow y= \tilde{\tilde{x}}(u) , \quad \forall u \in \cup_{z\in\R} B_{L^2}(\varphi(\cdot-z),\tilde{\tilde{\varepsilon}}_0) \; .
  $$
 We set $ \varepsilon_0=\tilde{\varepsilon}_0\wedge \tilde{\tilde{\varepsilon}}_0 $. By uniqueness it holds $ \tilde{\tilde{x}}\equiv \tilde{x}$ on $ B_{H^1}(\varphi(\cdot-z),\varepsilon_0) $ and thus $ \tilde{x}$  is also a $ C^1 $-function on $ \cup_{z\in\R} B_{H^1}(\varphi(\cdot-z),\varepsilon_0) $ equipped with the metric inducted by the $ L^2(\R)$-metric. 
 
 Now, according to \eqref{gff}, it holds $ \{\frac{1}{c}u(t,), t\in \R\} \subset \cup_{z\in\R} B_{H^1}(\varphi(\cdot-z),\varepsilon_0)$  so that we can define  the function $ x(\cdot)$ on $\R$ by setting $ x(t)=\tilde{x}(u(t)) $. By construction $ x(\cdot) $ satisfies \eqref{distxz}-\eqref{ort}. Moreover, \eqref{gf} together with \eqref{rd} ensure that for any $ c>0 $ and any $
 0<\varepsilon<c \varepsilon_0 $, it holds 
 \begin{equation}\label{vb1}
 \|\frac{1}{c}u(t)-\varphi(\cdot-x(t))\|_{H^1} \le (\frac{\varepsilon}{c})^2 + \sup_{|z|\le C_0(\frac{\varepsilon}{c})^2} \|\varphi-\varphi(\cdot-z))\|_{H^1} \lesssim  (\frac{\varepsilon}{c})^2+\sqrt{C_0}\,  \frac{\varepsilon}{c}
 \end{equation}
which proves \eqref{fg}.

  In view of \eqref{CH}, any solution $ u\in C(\R;H^1(\R)) $ of (C-H) satisfies $ u_t\in C(\R;L^2(\R)) $ and thus belongs to $C^1(\R;L^2(\R)) $. This ensures that 
   the mapping $ t\mapsto x(t)=\tilde{x}(u(t)) $ is of class $ C^1$ on $ \R $.  Setting  $ R(t,\cdot)=c \varphi (\cdot-x(t))$ and $ w=u-R$ and differentiating \eqref{ort} with respect to time we get 
  \begin{eqnarray}\label{vb}
   \int_{\R} w_t (\rho_{n_0}\ast\varphi)'(\cdot-x(t))&=& \dot{x}(t) \int_{\R} w\,  (\rho_{n_0}\ast\varphi)''(\cdot-x(t)) \nonumber \\
   & = & -\dot{x}(t) \int_{\R} \partial_x w\,  (\rho_{n_0}\ast\varphi)'(\cdot-x(t))\nonumber \\
   &=& (\dot{x}(t)-c)O(\|w\|_{H^1})+ c \, O(\|w\|_{H^1})\; .
  \end{eqnarray}
  Substituting $ u $ by $ w+R $ in \eqref{CH} and using that $ R $ satisfies
  $$
  \partial_t R +(\dot{x}-c) \partial_x R + R \partial_x R +(1-\partial_x^2)^{-1} \partial_x (R^2+R_x^2/2)=0 \; ,
  $$
  we infer that $ w$ satisfies 
  $$
  w_t -(\dot{x}-c) \partial_x R = -\frac{1}{2} \partial_x \Bigl( (w+R)^2-R^2\Bigr) -(1-\partial_x^2)^{-1} \partial_x \Bigl( (w+R)^2-R^2+\frac{1}{2} ((w_x+R_x)^2-R_x^2)\Bigr) \; .
  $$
  Taking the $ L^2$-scalar product of this last equality with $(\rho_{n_0}\ast \varphi)'(\cdot-x(t)) $ and using \eqref{vb} together with \eqref{fg} we get 
  $$
  \Bigl|(\dot{x}-c) \Bigl(\int_{\R}\,  \partial_x R \partial_x(\rho_{n_0}\ast \varphi)(\cdot-x(t))+c\, O(\|w\|_{H^1})\Bigr) \Bigr|\le  O(\|w\|_{H^1})\lesssim K c\, \varepsilon_0
  $$
  and \eqref{vb0} leads, by taking $ n_0 $ large enough and possibly decreasing the value of  $ \varepsilon_0>0 $ so that $ K\varepsilon_0 \ll 1 $,  to \eqref{estc}. 
  
 It remains to prove that \eqref{unic} holds for $ n_0\ge 0 $ large enough. For this we  notice that 
 $
 \int_{\R} \varphi' \varphi'(\cdot-y) = (1-y) e^{-y}
 $
 which ensures that for $ n_0\ge 0   $ large enough
 $$
 \frac{d}{dy} \int_{\R} \varphi  (\rho_{n_0} \ast \varphi)'(\cdot-y)=\int_{\R} \varphi'  (\rho_n \ast \varphi)'(\cdot-y) \ge \frac{1}{4}e^{-\frac{1}{2}}\mbox{ on } [-1/2,1/2] \; .
 $$
 Therefore $ y\mapsto  \int_{\R} \varphi  (\rho_{n_0} \ast \varphi)'(\cdot-y) $ is increasing on $[-1/2,1/2]$ and the proof is complete.
    \subsection{Proof of Proposition \ref{propasym}}
 Let $ u_0\in Y_+ $ satisfying \eqref{stab} with $ 0<\varepsilon<\frac{c}{2^8}$. First  we recall that, on account of  \eqref{fg} and Lemma  \ref{almostdecay},  the solution $u $ to (C-H), emanating from $ u_0$,  satisfies \eqref{monoJr}-\eqref{monoJl} with $ 0\le \gamma\le c $. 
 Let $ \{t_n\} \nearrow +\infty $. 
  Since, by  \eqref{estc}, $ \{x(t_n+\cdot)-x(t_n)\} $ is uniformly equi-continuous, Arzela-Ascoli theorem ensures that there exists a subsequence
   $ \{t_{n_k}\} \subset \{t_n\} $ and $ \tilde{x}\in C(\R) $ such that for all $ T>0 $,
   \begin{equation}\label{cvx}
   x(t_{n_k}+\cdot)-x(t_{n_k}) \tendsto{t\to+\infty} \tilde{x} \mbox{ in } C([-T,T]) \; .
   \end{equation}
   Now,  since $u(t_n) $ is bounded in $ Y_+$ . There exists $ \tilde{u}_0 \in Y_+ $ and a subsequence of $ \{t_{n_k}\} $  (that we still denote by $t_{n_k} $ to simplify the notation) such that 
\begin{eqnarray}
u(t_{n_k},\cdot +x(t_{n_k})) & \rightharpoonup & \tilde{u}_0 \mbox{ in } H^1(\R) \nonumber \\
u(t_{n_k},\cdot +x(t_{n_k})) & \to & \tilde{u}_0 \mbox{ in } H^1_{loc}(\R)  \label{vc} \\  
y(t_{n_k},\cdot +x(t_{n_k})) & \rightharpoonup \! \ast  & \tilde{y}_0= \tilde{u}_0 - \tilde{u}_{0,xx}  \mbox{ in } {\mathcal M}(\R) \; .\nonumber
\end{eqnarray}
Let $ \tilde{u}\in C(\R;Y_+)  $ be  the solution to \eqref{CH}  emanating from $ \tilde{u}_0 $. 
 On account of  \eqref{cvx} and part {\bf 3.} of Proposition \ref{WP} for any $ t\in \R $,
\begin{eqnarray}
u(t_{n_k}+t,\cdot +x(t_{n_k}+t)) & \rightharpoonup & \tilde{u}(t,\cdot+\tilde{x}(t)) \mbox{ in } H^1(\R), \ \label{weakcv}\\
u(t_{n_k}+t,\cdot +x(t_{n_k}+t)) & \to & \tilde{u}(t,\cdot+\tilde{x}(t)) \mbox{ in } H^1_{loc}(\R)   \label{strongcv}\
\end{eqnarray}
Moreover, for any  function $ \phi\in C_0(\R) $, it holds 
\begin{equation}
\dist{y(t_{n_k}+t,\cdot +x(t_{n_k}+t))}{\phi} \rightarrow  \dist{\tilde{y}(t,\cdot+\tilde{x}(t))}{\phi}  \label{weakcvy}\; ,
\end{equation}
where $ \tilde{y}=\tilde{u}-\tilde{u}_{xx} $. Indeed, on one hand, it follows from  part {\bf 3.} of Proposition \ref{WP}  that 
$$
\dist{y(t_{n_k}+t,\cdot +x(t_{n_k})+\tilde{x}(t))}{\phi} \rightarrow  \dist{\tilde{y}(t,\cdot+\tilde{x}(t))}{\phi} 
$$
and on the other hand, the uniform continuity of $ \phi $ together with \eqref{cvx} ensure that 
\begin{align*}
&\dist{y(t_{n_k}+t,\cdot +x(t_{n_k})+\tilde{x}(t))-y(t_{n_k}+t,\cdot +x(t_{n_k}+t))}{ \phi}  \\
&= \dist{y(t_{n_k}+t)}{\phi(\cdot-x(t_{n_k})-\tilde{x}(t))-\phi(\cdot-x(t_{n_k}+t))}\to 0 
\end{align*}
In view of \eqref{weakcv}  we infer that $ (\tilde{u}, \tilde{x}(\cdot)) $  satisfies   \eqref{ort} and \eqref{fg} 
with the same $ \varepsilon $ than $ (u,x(\cdot))$.   Therefore,  \eqref{defep} 
 forces    $ (\tilde{u}, \tilde{x}(\cdot)) $ to satisfy \eqref{gff} and the  uniqueness result in Lemma \ref{modulation} ensures that $ \tilde{x}(\cdot) $ is a $ C^1$-function  and  satisfies \eqref{estc}.

  The proof of the $Y$-almost localization of the asymptotic object $ \tilde{u} $ will  now proceed by contradiction. Let us first explain it briefly. 
  In the sequel, for $ 0<\gamma\le \frac{2c}{3} $ fixed, we call by G the conserved quantity 
  $$
   G(u)=E(u)+\gamma \dist{u-u_{xx}}{1} \; .
   $$
    If  $\tilde{u} $ is not $Y$-almost localized then $ \tilde{u} $ loses a certain amount of $ G $ close to $\tilde{x}(t) $ between $ 0 $ and  some $ T_0>0 $. By the convergence results \eqref{strongcv}-\eqref{weakcvy} we infer that for $ n $ large enough, $ u $ loses some fixed amount  of $ G $ close to $x(t)$ between $ t_n$ and $ t_n+T_0 $. By the conservation of $G$ on the whole line and the almost monotonicity of $ J_{\gamma,l}$ and $ J_{\gamma,r}$ this ensures that for some $ R>0 $ and $ \varepsilon_0>0 $,
  $$
  J_{\gamma,l}^R \Bigl(u(t_{n_k}+T_0, \cdot +x(t_{n_k}+T_0))\Bigr)\ge J_{\gamma,l}^R\Bigl(u(t_{n_k}, \cdot+x(t_{n_k}))\Bigr) + \varepsilon_0 \; .
  $$
  But by the almost monotonicity of $ J_{\gamma,l} $, taking $ \{t_{n_k}\}\subset \{t_n\} $ such that $ n_k\ge n_0 $ and $ t_{n_{k+1}}-t_{n_k} \ge T_0 $ we get 
  $$
  J_{\gamma,l}^R\Bigl(u(t_{n_k}, \cdot+x(t_{n_k}))\Bigr) \ge J_{\gamma,l}^R\Bigl(u(t_0,\cdot +x(t_0))\Bigr)+k \varepsilon_0/2 
  $$
  which contradicts the conservation of $ G$. 

Let us now  make this proof rigorously. For $ v\in Y $ and $ R>0 $, we separate  $ G(v) $ into  two parts : 
$$
G_o^{R}(v)=\dist{v^2+v_x^2+\gamma (v-v_{xx})}{1-\Psi(\cdot+R)+\Psi(\cdot-R)}=J_{\gamma,r}^R(v)+J_{\gamma,l}^R(v) \; ,
$$
which  almost ``localizes'' outside  the ball of radius $ R $ and 
$$
G_i ^{R}(v)=\dist{v^2+v_x^2+\gamma (v-v_{xx})}{\Psi(\cdot+R)-\Psi(\cdot-R)}=G(v)-G_o^R(v) \; ,
 $$
which almost ``localizes'' inside this ball.
 We first notice that it suffices  to prove that  for all $\varepsilon>0 $, there exists $ R_\varepsilon>0 $ such that
\begin{equation}\label{po}
G_o^{R_\varepsilon}\Bigl(\tilde{u}(t,\cdot+\tilde{x}(t))\Bigr)<\varepsilon \; , \quad \forall t\in\R \, .
\end{equation}
 \noindent
 Indeed if \eqref{po} is true for some $ (\varepsilon, R_\varepsilon) $ then $ (\tilde{u},\tilde{x}) $ satisfies \eqref{defloc} with $(\varepsilon/2, 2 R_\varepsilon) $.
 As indicated above, we prove \eqref{po} by contradiction. Assuming that \eqref{po} is not true,  there exists $ \varepsilon_0>0 $ such that for any $ R>0 $ there exists $ t_R\in \R $ satisfying 
  \begin{equation}
  G_o^{R}\Bigl( \tilde{u}(t_R, \cdot+\tilde{x}(t_R))\Bigr)\ge \varepsilon_0
  \end{equation}
  Let $ R_0>0 $ such that 
   \begin{equation}
  G_o^{R_0}\Bigl( \tilde{u}(0)\Bigr)\le \frac{\varepsilon_0}{10}
  \end{equation}
  and $ K_0 e^{-R_0/6}<\frac{\varepsilon_0}{10}$.  The conservation of $G $ then forces
  $$
   G_i^{R_0} (\tilde{u}(t_{R_0}, \cdot+\tilde{x}(t_{R_0}))\le G^{R_0}_i(\tilde{u}(0))-\frac{9}{10} \varepsilon_0 \; .
  $$
 Noticing that  $ \Psi(\cdot+R) -\Psi(\cdot-R) \in C_0(\R) $,  the convergence results \eqref{strongcv}-\eqref{weakcvy} ensure that for $ k\ge k_0$ with $ k_0 $ large enough,
  $$
  G_i^{R_0} (u(t_{n_k}+t_{R_0}, \cdot+x(t_{n_k}+t_{R_0})))\le G^{R_0}_i(u(t_{n_k}, \cdot+x(t_{n_k})))-\frac{4}{5} \varepsilon_0 \; .
  $$
  We first assume that $ t_{R_0}>0 $. By \eqref{monoJr}-\eqref{monoJl} and the conservation of $G $ this ensures that 
  \begin{equation}\label{bb}
  J_{\gamma,l}^{R_0}(u(t_{n_k}+t_{R_0}, \cdot +x(t_{n_k}+t_{R_0}))) \ge  J_{\gamma,l}^{R_0}(u(t_{n_k}, \cdot +x(t_{n_k}))) +\frac{7}{10} \varepsilon_0 \; .
  \end{equation}
  Now we take a subsequence $\{t_{n_k'}\}$ of $ \{t_{n_k}\} $ such that $ t_{n_{k+1}'}-t_{n_k'} \ge t_{R_0}$ and $ n_k'\ge n_{k_0} $. From \eqref{bb} and again \eqref{monoJl},  we get that for any $ k\ge 0 $, 
  $$
  J_{\gamma,l}^{R_0}(u(t_{n_k'}, \cdot +x(t_{n_k'})) )\ge  J_{\gamma,l}^{R_0}(u(t_{n_0'}, \cdot +x(t_{n_0'})) )+\frac{3}{5} k\,  \varepsilon_0 \tendsto{k\to +\infty} +\infty \; 
  $$
  that contradicts the conservation of $ G $ and thus proves the $Y$-almost localization of $\tilde{u} $.
  Finally, if $ t_{R_0} <0 $, then for $ k\ge k_0 $ such that  $ t_{n_k}>|t_{R_0}| $  we get in the same way 
  $$
   J_{\gamma,r}^{R_0}\Bigl(u(t_{n_k}, \cdot +x(t_{n_k}))\Bigr)\le  J_{\gamma,r}^{R_0}\Bigl(u(t_{n_k}-|t_{R_0}|, \cdot +x(t_{n_k}-|t_{R_0}|))\Bigr) -\frac{7}{10} \varepsilon_0 \; .
   $$
   As above, this implies  the existence of a subsequence $\{t_{n_k'}\}$ of $ \{t_{n_k}\} $ such that
   $$
  J_{\gamma,r}^{R_0}(u(t_{n_k'} \cdot +x(t_{n_k'})) )\le  J_{\gamma,r}^{R_0}(u(t_{n_0'}, \cdot +x(t_{n_0'}))) -\frac{3}{5} k\,  \varepsilon_0 \tendsto{k\to +\infty} -\infty \; .
  $$
   which also leads to a contradiction.\vspace*{2mm}\\
  {\bf Acknowledgements :} 
  The author is very grateful to Professor Yvan Martel for his encouragements, his careful reading of a first version of the manuscript  and having suggested the proof of   Lemma \ref{yvan}. He also thank Professor Masaya Maeda for pointing out a flaw in the version of Lemma \ref{modulation} stated in \cite{EL2}
   as well as  the anonymous Referees for valuable remarks.
Finally, the author  gratefully  acknowledges the hospitality and support of IHES, where part of this work was done during the program on {\it nonlinear waves} in spring 2016.
\vspace{3mm}\\
\noindent
 {\bf Conflict of Interest }: The author declares that he has no conflict of interest.

\end{document}